\documentclass[12pt]{amsart}
\usepackage[english]{babel}
\usepackage{amsmath,amssymb,amsthm,latexsym,amsfonts}
\usepackage[mathscr]{eucal}
\usepackage{graphicx}
\usepackage{tikz}
\theoremstyle{plain}
\newtheorem{lemma}{Lemma}
\newtheorem{theorem}[lemma]{Theorem}
\newtheorem{corollary}[lemma]{Corollary}
\newtheorem{proposition}[lemma]{Proposition}
\newtheorem{remark}{Remark}

\def\D{\Delta}

\def\G{\Gamma}
\def\g{\gamma}

\def\int{\textrm{int}}
\newcommand{\matr} [4] {\left(\begin{array}{@{}c@{\ }c@{}} #1 & #2 \\ #3 & #4 \\ \end{array} \right)}

\newcommand{\seifuno}[3]{\big(#1,({\scriptstyle #2},{\scriptstyle #3})\big)}

\newcommand{\seifdue}[5]{\big(#1,({\scriptstyle #2},{\scriptstyle #3}),
                       ({\scriptstyle #4},{\scriptstyle #5})\big)}

\newcommand{\bigu}[4]{\bigcup\nolimits_{{\tiny{\matr {#1} {#2} {#3} {#4}}}\phantom{\Big|}\!\!}}

\begin{document}

\title{$4$-colored graphs and knot/link complements}

\author{P. CRISTOFORI, E. FOMINYKH, M. MULAZZANI, V. TARKAEV}

\address{Paola CRISTOFORI, Dipartimento di Scienze Fisiche, Informatiche e Matematiche, Universit\`a di Modena e Reggio Emilia, Italy}

\email{paola.cristofori@unimore.it}

\address{Evgeny FOMINYKH, Chelyabinsk State University and Krasovskii Institute of Mathematics and Mechanics, Russia}

\email{efominykh@gmail.com}

\address{Michele MULAZZANI, Dipartimento di Matematica and ARCES,  Universit\`a di Bologna, Italy}

\email{mulazza@dm.unibo.it}

\address{Vladimir TARKAEV, Chelyabinsk State University and Krasovskii Institute of Mathematics and Mechanics, Russia}

\email{v.tarkaev@gmail.com}

\begin{abstract}
A representation for compact $3$-manifolds with  non-empty non-spherical boundary via $4$-colored graphs (i.e. $4$-regular graphs endowed with a proper edge-coloration with four colors) has been recently introduced by two of the authors, and an initial classification of such manifolds has been obtained up to 8 vertices of the representing graphs.
Computer experiments show that the number of graphs/manifolds grows very quickly as the number of vertices increases. As a consequence, we have focused on the case of orientable $3$-manifolds with toric boundary, which contains the important case of complements of knots and links in the $3$-sphere.

In this paper we obtain the complete catalogation/classification of these $3$-manifolds up to $12$ vertices of the associated graphs, showing the diagrams of the involved knots and links.
For the particular case of complements of knots, the research has been extended up to $16$ vertices. 

\bigskip
\medskip


\smallskip

\end{abstract}

\maketitle

\section{Introduction and preliminaries}\label{sec1}

The representation of closed $3$-manifolds by $4$-colored graphs has been independently introduced in the seventies by M. Pezzana and in the eighties by S. Lins (see \cite{[FGG]} and \cite{[Li]}), by using dual constructions. A $4$-colored graph is a regular graph of degree four, endowed with a proper edge-coloration (i.e., adjacent edges have different colors) with four colors, and it represents a closed $3$-manifold if and only if a certain combinatorial condition is satisfied.
This type of representation has been very fruitful with regard to the catalogation of closed $3$-manifolds in terms of increasing complexity, which in this context means number of vertices of the representing graph (see \cite{[Li]}, \cite{[BCrG$_1$]}, \cite{[CC]} and \cite{[CC1]}).

The extension of the representation by $4$-colored graphs to $3$-manifolds with boundary has been performed in \cite{[CM]}, where any $4$-colored graph is associated to a compact $3$-manifold with (possibly empty)  boundary without spherical  components, this correspondence being surjective on the whole class of such manifolds. 

As a consequence, an efficient computer aided catalogation of $3$-manifolds with boundary can be performed by this tool. Some computational results have been already obtained in  \cite{[CM]}, namely the classification of all manifolds representable by a graph of order $\leqslant 8$ in the orientable case and $\leqslant 6$ in the non-orientable one.

In the present paper our main goals are the creation of:
\begin{itemize}
\item[(a)] the census of all non-isomorphic contracted  bipartite $4$-colored graphs (without $2$-dipoles) up to some order, representing compact orientable $3$-manifolds with (possibly disconnected) toric boundary;
 \item[(b)] the census of the prime $3$-manifolds represented by the above graphs;
 %\item[(c)] the census of the knots and links in the $3$-sphere whose complements are among the manifolds of the above census. 
\item[(c)] the catalogue of all manifolds in the above census which are complements of knots and links in the $3$-sphere, also giving, for each of such manifolds, the diagram of a corresponding
knot/link. 
\end{itemize}

%In Section 2 we describe the construction which associates a compact $3$-manifold to any $4$-colored graph. 

For (a), we wrote a script that, for each positive integer $p$, enumerates all non-isomorphic contracted $4$-colored graphs with $2p$ vertices and no $2$-dipoles, representing compact bordered $3$-manifolds, possibly fixing the topological type of their boundary. In Section~\ref{sec3}, after describing the enumeration process, we give computational results for $2p \leqslant 12$ with no restriction on the orientability and the topological type of the boundary (see Table~\ref{tab:1}), and for $2p \leqslant 16$ in the case of graphs representing orientable manifolds with toric boundary (see Table~\ref{tab:2} in the first two rows).

%In order to get the census in (b) from the generated graphs representing orientable manifolds with toric boundary, three steps were necessary. 

In order to get the census of (b) from the one of (a), three steps were necessary. 
The first one consisted in removing from the census non-rigid $4$-colored graphs, since they were not minimal among those representing the same manifold or they represented non-prime manifolds (see Corollary \ref{rigid}).

In the second step, this ``reduced'' census (see Table~\ref{tab:2} in the third row) was processed by \verb"3-Manifold Recognizer" \cite{[Recognizer]}, a computer program for studying $3$-manifolds,  written by V. Tarkaev according to an algorithm elaborated by S. Matveev and other members of the Chelyabinsk topology group. For each $4$-colored graph $\G$ this program constructed a triangulation of the manifold $M_{\G}$ associated to $\G$ and thus obtained a census of compact orientable $3$-manifolds with toric boundary. 
Then \verb"3-Manifold Recognizer" removed non-prime manifolds from the census and recognized which manifolds were graph manifolds of Waldhausen (see \cite{[Wa]}). 
It means that for each graph manifold the program described its canonical Seifert blocks and the way how they are glued together (see Subsections \ref{Seifert} and \ref{graph-manifolds}). 
Then, according to Waldhausen classification of graph manifolds, we removed duplicates from the census. For the manifolds $M_{\G}$ that were not recognized by \verb"3-Manifold Recognizer", the program returned the isomorphism signature\footnote{We recall that the isomorphism signature is an improvement by B. Burton \cite{[Burton]} of the (non-canonical) dehydration sequences \cite{[CHW]}. It is a complete invariant of the combinatorial isomorphism type of a triangulation.} of the triangulation corresponding to $\G$.

In the third step the census of (b) was completed by using \verb"SnapPy" \cite{[SnapPy]}, a computer program for studying the topology of $3$-manifolds, written by M. Culler and N. Dunfield using the \verb"SnapPea" kernel by J. Weeks, with contributions from many others. 
More precisely, the isomorphism signatures were entered into \verb"SnapPy", which identified all still unrecognized manifolds, except two, as hyperbolic manifolds in its orientable cusped census or in its censuses of Platonic manifolds (see Section~\ref{sec4}). 
The two exceptional manifolds were identified subsequently in the catalogue of (c) as complements of the links $L10n111$ and $L12n2205$ in the $3$-sphere. 

For the census in (c), we used the Thistlethwaite link table up to $14$ crossings distributed with \verb"SnapPy" and in some cases the recognized structure of graph manifolds. It turned out that all $3$-manifolds in our census except two were complements of links in the $3$-sphere (see Table~\ref{tab:3}). For each manifold of this kind we also displayed the diagram of a corresponding link (see Figures~\ref{links1} and \ref{links2}). When more than one link was available a prime one was preferred (identified by its name in the Thistlethwaite link table). 

%When a  $4$-colored graph is bipartite, special configurations of its edges gives some information about the minimality of the graph and the reducibility of the associated orientable manifold.

%Of the 124 orientable tetrahedral manifolds with at most 12 tetrahedra, 27 are homology links and SnapPy identified 13 of them with link exteriors in its census.

%%% ---------------------------------------------------------------------
\section{From $4$-colored graphs to compact $3$-manifolds}
\label{sec2}

Let $\G$ be a finite graph which is $4$-regular (i.e., each vertex has degree four), possibly with multiple edges but with no loops.  A map $\g:E(\G)\to \D=\{0,1,2,3\}$ is called a \textit{$4$-coloration} of $\G$ if adjacent edges have different colors. An edge of $\G$ colored by $c\in\D$ is also called a \textit{$c$-edge}.

A \textit{$4$-colored graph} is a connected $4$-regular graph equipped with a $4$-coloration. 
It is easy to see that any $4$-colored graph has even order.
If $\D'\subseteq\D$, a
\textit{$\D'$-residue} (as well as a \textit{$\vert\D'\vert$-residue}) of $\G$
is any connected component of the subgraph $\G_{\D'}$ of $\G$
containing exactly all $c$-edges, for each $c\in\D'$. Of course,
$0$-residues are vertices, $1$-residues are edges and $2$-residues are
bicolored cycles. %For $i,j\in\D$, with $i\ne j$, we denote by $g_{i,j}$ the number of $\{i,j\}$-residues of $\G$. Moreover, for each $c\in\D$, the number of $3$-residues corresponding to the colors of $\hat c=\D-\{c\}$ will be denoted by $g_c$.
For each $c\in\D$, we denote by $g_c$ the number of $3$-residues corresponding to the colors of $\hat c=\D-\{c\}$.

A  $4$-colored graph $\G$ induces a compact connected $3$-manifold $M_{\G}$
without spherical boundary components  via the following construction.

%By attaching disks to all $2$-residues 
By attaching a disk to each $2$-residue of $\G$ (considered as a $1$-dimensional cellular complex) we obtain a $2$-dimensional polyhedron $P_{\G}$. 
Each $3$-residue of $\G$, with its associated disks, forms a closed connected surface $S$ embedded in  $P_{\G}$. If $S$ is a $2$-sphere the residue is called \textit{ordinary}, otherwise it is called \textit{singular}.

Then, for any $3$-residue $S$, fill it with a $3$-ball if it is ordinary, and just thicken it by attaching $S\times I$ along $S\times \{0\}$ if it is singular.  At the end of this process, a compact connected $3$-manifold $M_{\G}$ with non-spherical boundary components is obtained.
We will say that $\G$ \textit{represents} $M_{\G}$, as well as that  $M_{\G}$ is  \textit{associated} to $\G$. Obviously isomorphic $4$-colored graphs\footnote{Two $4$-colored graphs $\G'$ and $\G''$, with coloration $\g'$ and $\g''$ respectively, are isomorphic if there exist a graph isomorphism $\phi$ between $\G'$ and $\G''$ and a permutation $\sigma$ of $\D$ such that $\g''\circ\phi=\sigma\circ\g'$.} represent homeomorphic $3$-manifolds.

For closed $3$-manifolds the construction reduces to the one introduced by Lins (see \cite{[Li]}), and it is dual to the one introduced by Pezzana (see \cite{[FGG]}). So, the novelty of this construction is that it works also in case of $3$-manifolds with boundary, and  $\G\mapsto M_{\G}$  is a surjective map  from the whole class of $4$-colored graphs to the whole class of compact $3$-manifolds with (possibly empty) boundary without spherical components.

In fact, the following result is proved in \cite{[CM]}.

\begin{proposition} If $M$ is a compact $3$-manifold with (possibly empty) boundary without spherical components, then there exists a $4$-colored graph $\G$ such that $M$ is homeomorphic to $M_{\G}$. Moreover, $M_{\G}$ is orientable if and only if $\Gamma$ is bipartite.
\end{proposition}

Since two compact $3$-manifolds are homeomorphic if and
only if (i) they have the same number of spherical boundary components
and (ii) they are homeomorphic after capping off by balls these
components, there is no loss of generality in
studying compact $3$-manifolds without spherical boundary
components. In the following with the term ``manifold''  we will always refer to such type of $3$-manifolds.

\begin{figure*}
\includegraphics[scale=0.4]{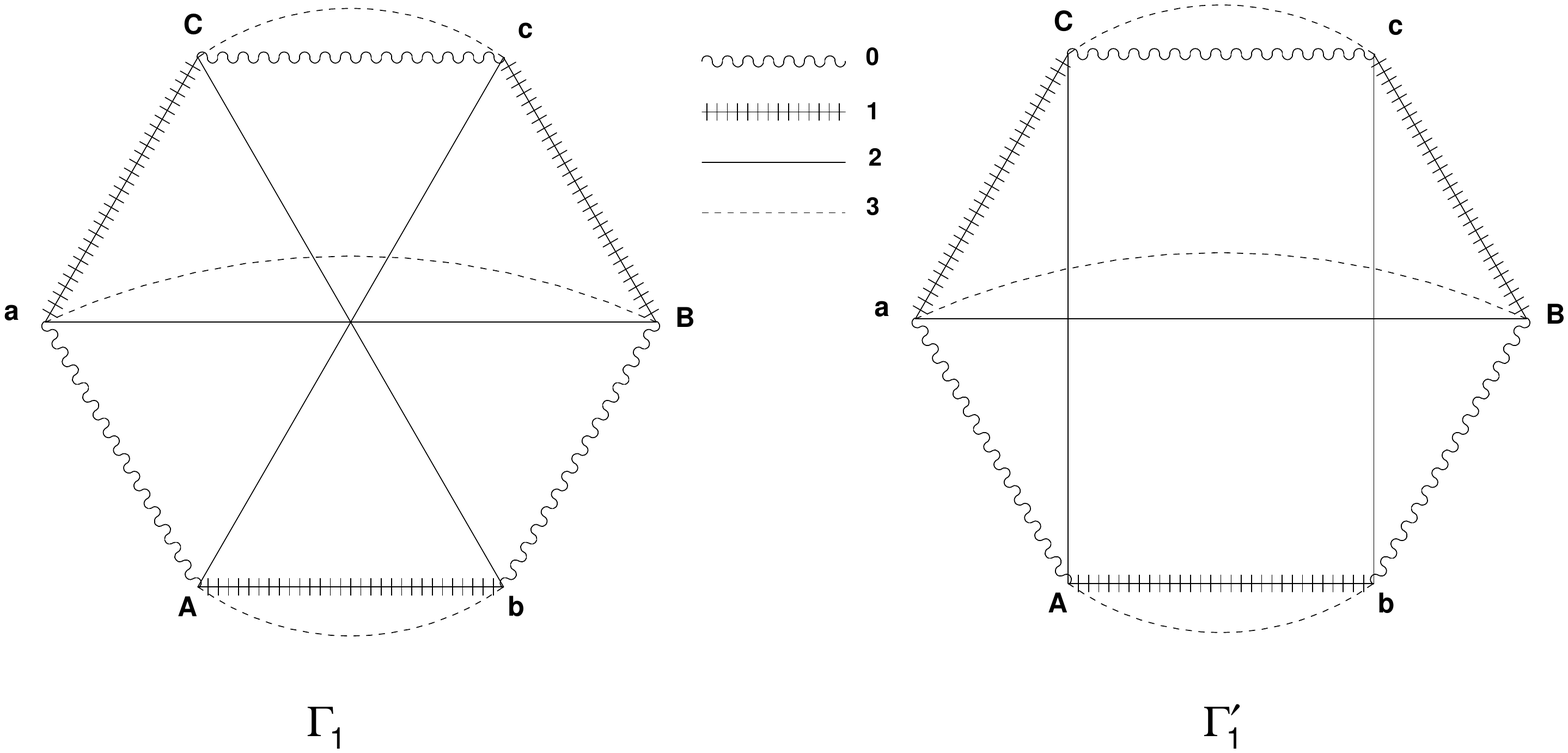}
\caption{The $4$-colored graphs $\G_1$ and and $\G_1'$, representing the genus one orientable and non-orientable handlebody respectively.}
\label{fig: Fig1}
\end{figure*}

In \cite{[CM]}, $4$-colored graphs representing important families of compact $3$-manifolds, such as handlebodies (both orientable and non-orientable, see Fi\-gure~\ref{fig: Fig1}), and $S\times I$, where $S$ is any closed surface  (either orientable or non-orientable) are shown.

As it usually happens in any representation technique of $3$-manifolds (triangulations, spines, branched coverings, surgery on links, etc...), any compact $3$-manifold admits representations by infinitely many $4$-colored graphs. So, moves connecting non-isomorphic graphs representing the same $3$-manifold are important tools in the theory. 

Let $v$ and $v'$ be adjacent vertices of a $4$-colored graph $\G$ of order $> 2$, and let $\theta$ be the subgraph of $\G$ consisting of $v,v'$ and all the $h$ edges connecting them. Then $\theta$ is called an  \textit{$h$-dipole} if  $v'$ and $v''$ belong to different $(\D -\{c_1,\ldots,c_h\})$-residues
of $\G$, where $\{ c_1,\ldots, c_h\}$ are the colors of the $h$ edges of $\theta$.

By \textit{cancelling} $\theta$ \textit{from} $\G$, we mean to remove $\theta$ and to paste together the hanging edges according to their colors, thus obtaining a new $4$-colored graph $\G^\prime$. Conversely, $\G$ is said to be obtained from $\G^\prime$ by \textit{adding} $\theta$.
An $h$-dipole $\theta$ is called \textit{proper} when $\G$ and $\G^\prime$ represent the same manifold.

\begin{proposition}\label{proper}  \cite{[CM]} An $h$-dipole $\theta$ in a $4$-colored graph $\G$ is proper if and only if one of the following conditions holds:
\begin{itemize} \item $h>1$;
\item $h=1$ and at least one of the $\hat c_1$-residues containing the two vertices of $\theta$ is ordinary.
\end{itemize}
\end{proposition}

In the closed case all dipoles are proper and Casali proved in \cite{[C$_1$]} that dipole moves are sufficient to connect different $4$-colored graphs representing the same manifold. A similar fact is no longer true in our context; in fact dipole moves do not change the singular colors\footnote{By a \textit{singular color} of a $4$-colored graph $\G$, we mean a color $c\in\D$ such that at least one of the $\hat c$-residues of $\G$ is singular.} of the involved graph, and for any $3$-manifold with disconnected boundary it is easy to find two different
graphs representing it, the first one with only a singular color and the second one with (at least) two singular colors.

A $4$-colored graph $\G$  is said to be  \textit{contracted} if either $g_c=1$ or all $\hat c$-residues of $\G$ are singular, for each $c\in\D$.
Note that a $4$-colored graph is contracted if and only if it does not contain any proper $1$-dipole, and a contracted $4$-colored graph contains no $3$-dipoles. Moreover, by cancelling proper $1$-dipoles, any $4$-colored graph can be reduced to a contracted one representing the same manifold. So, any compact $3$-manifold admits representations by contracted $4$-colored graphs. 

A $4$-colored graph $\G$ is called {\it minimal} if there exists no graph with fewer vertices representing the same manifold. Of course, a graph containing a proper dipole is not minimal. In particular, any minimal graph is contracted.

%%% ---------------------------------------------------------------------

\section{First computational results}\label{sec3} 

The combinatorial nature of colored graphs makes them particularly suitable for computer manipulation. In particular, it is possible to generate catalogues of $4$-colored graphs for increasing number of vertices, in order to analyze the represented compact $3$-manifolds.
We focus our attention on $3$-manifolds with non-empty boundary, since the catalogation  of graphs associated to closed manifolds (also called \textit{gems}, and  \textit{crystallizations} in the contracted case), has been already performed up to 30 vertices (see \cite{[BCrG$_1$]} and \cite{[CC]}).

\subsection{Description of the enumeration process} \mbox{}

By the results of Section~\ref{sec2}, without loss of generality we can restrict the catalogues to contracted graphs with no $2$-dipoles. 
Therefore, given a positive integer $p$, the catalogue of contracted $4$-colored graphs with $2p$ vertices and no $2$-dipoles is generated algorithmically in the following way. 
 
First of all consider the set of all (possibly disconnected) $2$-colored graphs (i.e., disjoint union of $\{0,1\}$-colored cycles) with $2p$ vertices.
For any element of this set, add the $2$-edges in all possible way in order to obtain, by checking  bipartition and computing the Euler characteristic of the connected components, the set $\mathcal S^{(2p)}$ of all (possibly disconnected) $3$-colored graphs $\G$ with $2p$ vertices, such that each connected component of $\G$ represents a surface of positive genus. For the values of $p$ from $1$ to $6$, the cardinality of $\mathcal S^{(2p)}$ is $0,1,3,14,71$ and $553$, respectively.

Afterwards, the $3$-edges are added to each graph of $\mathcal S^{(2p)}$ in all possible ways giving rise to a contracted $4$-colored graph without $2$-dipoles.

In order to avoid the presence of isomorphic graphs in the catalogue, the concept of code of a $4$-colored graph is an useful tool.
The \textit{code} of a $m$-bipartite $4$-colored graph\footnote{A $4$-colored graph $\G$ is called $m$-bipartite,  if $m$ is the largest integer such that all $m$-residues of $\G$ are bipartite.} $\G$ with $2p$ vertices is a ``string'' of length  $(7-m)p$, which
completely describes both combinatorial structure and coloration of $\G$ (see 
\cite{[CG]} for details). The importance of the code for representing $4$-colored graphs relies on the fact that two $4$-colored graphs are isomorphic if and only if they have the same code. As a consequence, by using codes we can produce catalogues containing only graphs which are pairwise non-isomorphic.

Table~\ref{tab:1} shows the number of non-isomorphic contracted $4$-colored graphs without $2$-dipoles and inducing manifolds with non-empty boundary, where $C^{(2p)}$ (resp. $\tilde{C}^{(2p)}$) refers to  bipartite (resp. non-bipartite) graphs of order $2p$, with $1\leqslant p\leqslant 6$.

\bigskip

\begin{table}[htp]
\caption{$4$-colored graphs with $2p$ vertices and no dipoles representing $3$-manifolds with boundary}
\label{tab:1}
 \begin{tabular}{|c||c|c|c|c|c|c|} 
  \hline
   {\bf 2p } & 2 & 4 & 6 & 8 & 10 & 12\\
 \hline\hline \ & \ & \ & \ & \ & \ & \ \\
  {\bf $C^{(2p)}$} & 0 & 0 & 2 & 4 & 57 & 902\\ 
\hline \ & \ & \ & \ & \ & \ & \ \\
 {\bf $\tilde{C}^{(2p)}$} & 0 & 1 & 6 & 90 & 3967 & 395877\\
   \hline \end{tabular}
\end{table}

\bigskip

With regard to the classification of the manifolds represented by the above graphs, some results concerning cases with small number of vertices have been already established:

\vbox{
\begin{proposition}   \cite{[CM]}
\begin{itemize}
\item There exist exactly five non-closed compact orientable $3$-manifolds admitting a representation by a $4$-colored graph of order $\leqslant 8$.
\item There exist exactly seven non-closed compact non-orientable $3$-manifolds admitting a representation by a $4$-colored graph of order $\leqslant 6$.
\end{itemize}
\end{proposition}
}

As it is clear from Table~\ref{tab:1}, the number of generated graphs quickly increases with the number of vertices and becomes very large even in the initial segment of the catalogues. So we carried on our research taking care of the case with  (possibly disconnected) toric boundary, which includes the important case of complements of knots or links in the $3$-sphere. In this context, orientable manifolds are of particular interest. 
Notice that the five orientable manifolds of the previous proposition are in fact knot/link complements. 
Namely, the two manifolds represented by graphs with six vertices are the complements of the trivial knot and the Hopf link respectively, whereas the three manifolds represented by graphs with eight vertices are the complements 
of the links $L6n1$, $L8n8$ and  $L8n7$ respectively (see Table~\ref{tab:3} and Figure~\ref{links1}), according to the notation of the Thistlethwaite link table up to $14$ crossings distributed with \verb"SnapPy".
The Thistlethwaite name of a link is of the form $L[k]a[j_1]$ or $L[k]n[j_2]$, depending on whether the link is alternating or not. Here $k$ is the crossing number and $j_1, j_2$ are archive numbers assigned to each $(a, k)$, $(n, k)$ pair, respectively. For example, $L5a1$ is the first alternating link with $5$ crossings while $L6n2$ is the second non-alternating link with $6$ crossings.

At the end of this section, Table~\ref{tab:2} lists in the first two rows the number of non-isomorphic bipartite contracted $4$-colored graphs without $2$-dipoles representing manifolds with toric boundary, indicated by $C^{(2p)}_t$ and the special cases of connected toric boundary, indicated by $C^{(2p)}_{tc}$.
 
\subsection{Irreducible manifolds and rigid graphs} \mbox{}

When a  $4$-colored graph is bipartite, special configurations of its edges give some information about the minimality of the graph and the reducibility of the associated orientable manifold.

A {\it $\rho_i$-pair} of a bipartite $4$-colored graph $\G$, for $2\leqslant i\leqslant 3$, is a pair of two equally colored edges $\{e,f\}$ sharing exactly $i$ $2$-residues of $\G$.

The \textit{switching} of a $\rho_i$-pair $\{e,f\}$ is the operation of replacing the edges of the pair,  by two new  edges $e'$ and $f'$, with the same color as the old ones, connecting the endpoints of $e$ and $f$ in the only way which preserves the bipartition of $\G$ (see Figure~\ref{fig: Fig2} for the case $i=3$).
Each $2$-residue shared by the two edges of the pair splits into two cycles and the new bipartite $4$-colored graph $\G'$ obtained after the switching may be either connected or disconnected with two components. In particular, when $i=2$ the new graph $\G'$ is always connected.

\begin{figure*}
\includegraphics[scale=0.75]{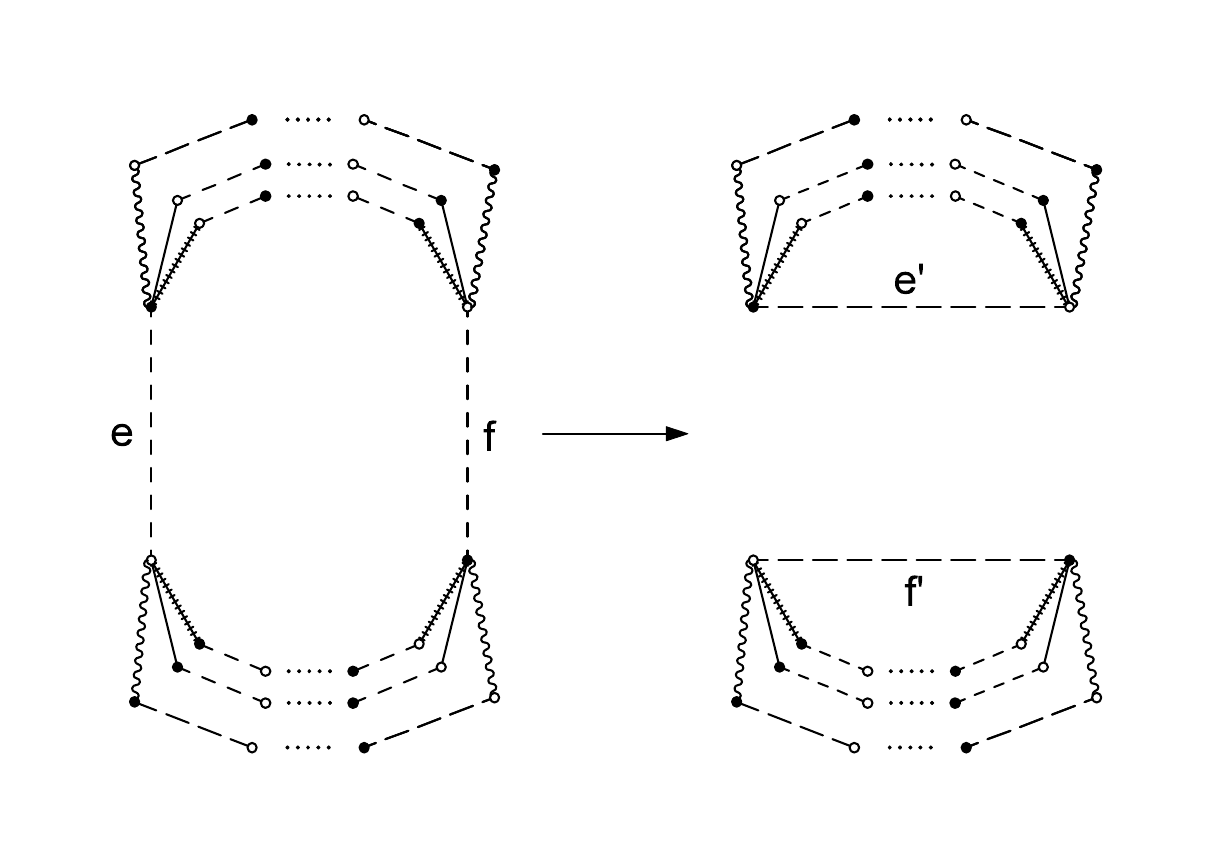}
\caption{Switching a $\rho_3$-pair.} 
\label{fig: Fig2}
\end{figure*}

First consider the effect of $\rho_2$-pair switchings. Let $\{e,f\}$ be a $\rho_2$-pair of $\G$, and let $c,j,k$ be the three colors of the $2$-residues shared by $e$ and $f$, we call  $\{e,f\}$ {\it good} if after the switching the $\{c,j,k\}$-residue containing $e$ and $f$ splits into two connected components and at least one of them is ordinary. The definition relies on the following lemma which extends to manifolds with non-empty boundary the corresponding result established by Lins for the closed case (\cite{[Li]}):

\begin{lemma}  \label{rho2}
Let $\G$ be a bipartite $4$-colored graph with a $\rho_2$-pair, and denote by $\G'$ the $4$-colored graph obtained from $\G$ by switching the pair. Then $\G$ and $\G'$ represent the same $3$-manifold if and only if the pair is good. In this case $\G$ is not minimal.
\end{lemma}

\begin{proof} 
Let $e,f$ be the edges of the pair, $e',f'$ be the new edges created by the switching, and let $c$ be their color. The $\{c,j,k\}$-residue involved in the switching, considered with its associated disks, is a closed orientable surface $S$. The switching operates on $S$ as follows: (i) split  each one of the two disks containing both $e$ and $f$ by cutting them along a properly embedded arc $\alpha_i$, with $i\in \{j,k\}$, connecting an internal point $E$ of $e$ with an internal point $F$ of $f$; (ii) push the edges $\alpha'_i$ and $\alpha''_i$ obtained from the cutting until both $e$ and $f$ disappear; (iii) then indentify $\alpha'_j,\alpha'_k$ (resp. $\alpha''_j,\alpha''_k$) to a single $c$-colored edge $e'$ (resp. $f'$). 
This operation is topologically equivalent to cutting $S$ along a suitable closed curve, and filling with two disks the two circles arising from the cutting. Consequently, if $\{e,f\}$ is not good then $\partial M_\G\neq\partial M_{\G'}$, since either $S$ decreases its genus or it splits into two components of positive genus.
Conversely, suppose that $\{e,f\}$ is good, then the switching can be factorized by a sequence of proper dipole moves as shown in Figure~\ref{fig: ro2}. Note that $D'$ is a $2$-dipole, hence it is always proper and its cancellation does not change the represented manifold, while $D$ is a proper $1$-dipole if and only if $\{e,f\}$ is good. In this case, $\G'$ is not contracted since it has at least two $\{c,j,k\}$-residues which are not both singular. As a consequence, $\G'$ is not minimal and therefore $\G$ is not minimal either, since it has the same order as $\G'$.
\end{proof}

\begin{figure*}
\includegraphics[scale=0.65]{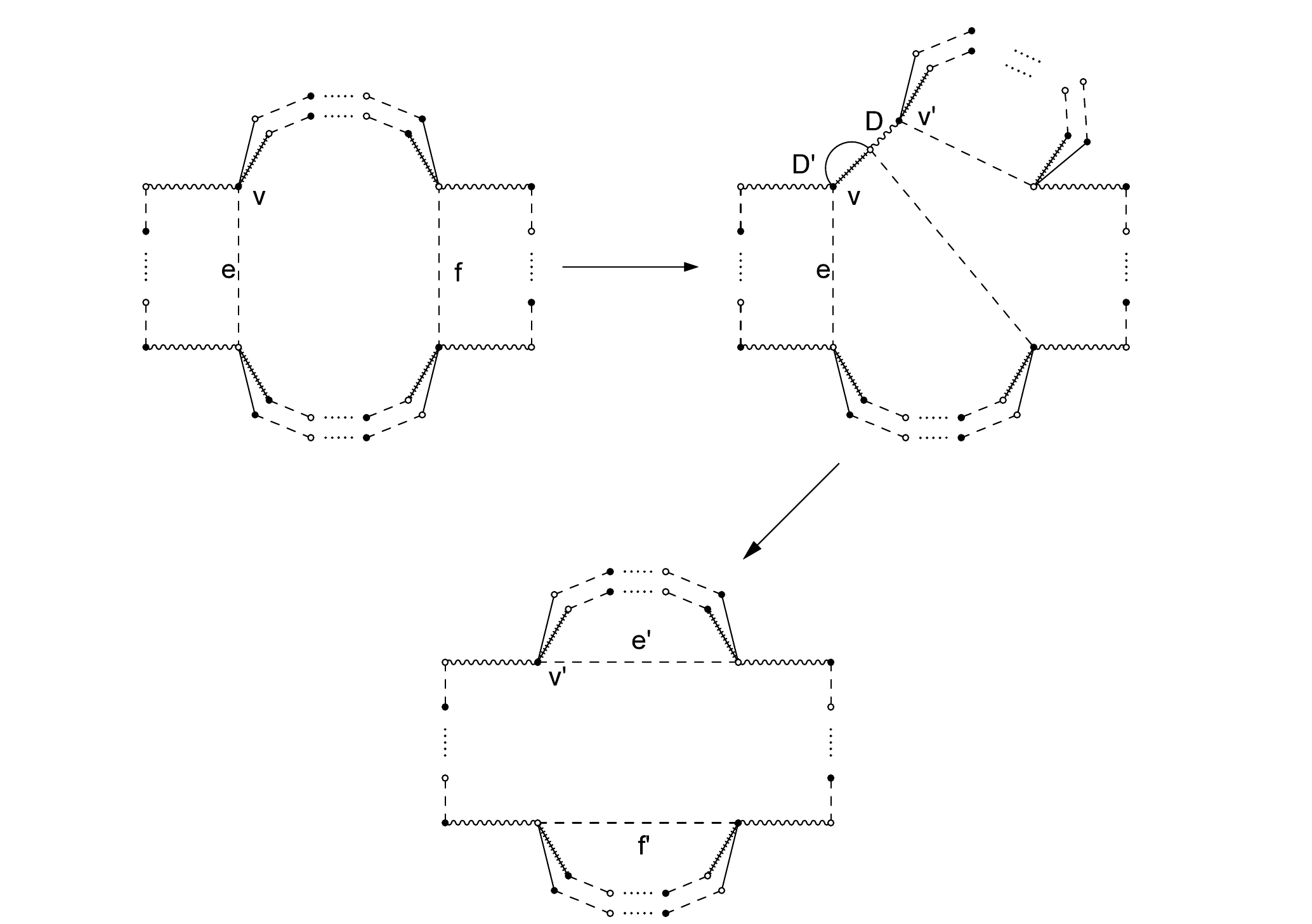}
\caption{Factorization of a $\rho_2$-pair switching} 
\label{fig: ro2}
\end{figure*}

%\vskip 8pt

\begin{corollary} \label{minimal}
Any minimal $4$-colored graph representing a compact orientable $3$-manifold contains no good $\rho_2$-pair.
\end{corollary}

By Corollary~\ref{minimal}, graphs containing good $\rho_2$-pairs can be ignored in a process of catalogation of compact orientable $3$-manifolds via $4$-colored graphs. Observe that a graph without good $\rho_2$-pairs is automatically without $2$-dipoles.

\medskip

Let us consider now the effect of switching $\rho_3$-pairs. 

It is easy to see that the two edges of a $\rho_3$-pair belong to the same three $3$-residues; the number $r$ of the ordinary ones is called the  \textit{index} of the $\rho_3$-pair. In the closed case all $\rho$-pairs are of index three, and in the case of connected boundary it is $r\geqslant 2$. 

In most cases, when the switching is performed on a $\rho_3$-pair of index $r\geqslant 2$, there is a strong topological relation between the manifold associated to $\G$ and the manifold(s) associated to $\G'$, as shown by the following lemma.

\begin{lemma} \label{rho3}
 Let $\G$ be a bipartite $4$-colored graph with a $\rho_3$-pair of index $r\geqslant 2$, and denote by $\G'$ the (possibly disconnected) $4$-colored graph obtained from $\G$ by switching the pair. 
\begin{itemize}
\item  [(i)] If $r=3$ and $\G'$ is disconnected with components $\G_1,\G_2$, then $M_{\G}=M_{\G_1}\#M_{\G_2}$.
\item  [(ii)] If $r=3$ and $\G'$ is connected, then  $M_{\G}=M_{\G'}\#(S^2\times S^1)$.
\item  [(iii)] If $r=2$ and $\G'$ is disconnected with components $\G_1,\G_2$, then either  $M_{\G}=M_{\G_1}\#M_{\G_2}$ or $M_{\G}=M_{\G_1}\#_{\partial}M_{\G_2}$.
\item  [(iv)] If $r=2$, $\G'$ is  connected and $\partial(M_{\G'})$ has the same number of components as  $\partial(M_{\G})$, then  either $M_{\G}=M_{\G'}\#(S^2\times S^1)$ or  $M_{\G}=M_{\G'}\#_{\partial} (D^2\times S^1)$.
\item  [(v)] If $r=2$, $\G'$ is  connected and $\partial(M_{\G'})$ has fewer components than  $\partial(M_{\G})$, then  $M_{\G}=M_{\G'}\# (D^2\times S^1)$. 
\item  [(vi)] If $r=2$, $\G'$ is  connected and $\partial(M_{\G'})$ has more components than  $\partial(M_{\G})$, then  $M_{\G}$ is boundary-reducible. 
\end{itemize}
In particular, if $r=2$ and $\partial(M_{\G})$ is a union of tori, then  either $M_{\G}=M_{\G_1}\#M_{\G_2}$ or  $M_{\G}=M_{\G'}\#(S^2\times S^1)$ or $M_{\G}=M_{\G'}\# (D^2\times S^1)$. 
\end{lemma}

\begin{proof} Let $e,f$ be the edges of the pair, $e',f'$ be the new edges created by the switching, and let $c$ be their color. 
In terms of $2$-skeleton the switching operates as follows: (i) split  each one of the three disks containing both $e$ and $f$ by cutting them 
along a properly embedded arc $\alpha_i$, with $i\in \D - \{c\}$, connecting an internal point $E$ of $e$ with an internal point $F$ of $f$; 
(ii) push the edges $\alpha'_i$ and $\alpha''_i$ obtained from the cutting until both $e$ and $f$ disappear; (iii) then indentify 
$\alpha'_i,\alpha'_j,\alpha'_k$ (resp. $\alpha''_i,\alpha''_j,\alpha''_k$) to a single $c$-colored edge $e'$ (resp. $f'$), 
where $  \{i,j,k\}=\D - \{c\}$. 
The cuttings of the disks extend to  $M_{\G}$ in the following way. For each $i\in   \D - \{c\}$, let $S$ be the $\hat i$-residue containing $e$ and 
$f$, considered with its associated disks. If $S$ is the $2$-sphere, the attached $3$-ball is cut along a properly embedded disk with boundary $\alpha_j\cup\alpha_k$, otherwise the attached 
$S\times I$ is cut along an annulus $(\alpha_j\cup\alpha_k)\times I$, where $ \{j,k\}= \D - \{c,i\}$. 

(i,ii) If $r=3$, the cutting is performed along three embedded disks whose union is an embedded sphere.  So (i) and (ii) are proved.

If $r=2$, the cutting is performed along two disks and an annulus, whose union is an embedded disk $D$ with boundary belonging to a boundary 
component of $M_{\G}$, which is %an orientable surface $S$ of genus $g>0$, and it is 
cut along $\partial D$. Then the arguments of the proof of Corollary 2.4.4 of \cite{[Ma]} apply, considering that in our case the genus $g$ of $S$ may be 
greater than one. 

(iii) If $\G'$ is disconnected and $S$ is non-trivially disconnected by the cutting into two components $S',S''$ then 
$M_{\G}=M_{\G_1}\#_{\partial}M_{\G_2}$, where the boundary connected sum is performed along $S'$ and $S''$ (of course this case cannot occur if the genus of $S$ is 1). If $\G'$ is disconnected and $S$ trivially disconnects into two components $S',S''$, where $S''$ is a disk, then $M_{\G}=M_{\G_1}\#M_{\G_2}$, since $D\cup S''$ is a $2$-sphere  (for $g=1$ this is Case 1 of \cite{[Ma]}).

(iv)  If $\G'$ is  connected and $\partial(M_{\G'})$ has the same number of components as  $\partial(M_{\G})$ then either $S$ is trivially 
disconnected by the cutting or $g>1$ and $S$ is not disconnected by the cutting. In the first case  $M_{\G}=M_{\G'}\#(S^2\times S^1)$ 
(for $g=1$ this is Case 2 of \cite{[Ma]}) and in the second one $M_{\G}=M_{\G'}\#_{\partial} (D^2\times S^1)$.

(v) If $\G'$ is  connected and $\partial(M_{\G'})$ has fewer components than  $\partial(M_{\G})$, then $g=1$ and $S$ is not disconnected by 
the cutting. So $M_{\G}=M_{\G'}\# (D^2\times S^1)$ (this is Case 3 of \cite{[Ma]}).

(vi) If $\G'$ is  connected and $\partial(M_{\G'})$ has more components than  $\partial(M_{\G})$, then $g>1$ and $S$ is disconnected by the 
cutting into two components of positive genus. Therefore  $\partial(M_{\G})$ is compressible (i.e., $M_{\G}$ is boundary-reducible). 
\end{proof} 

\medskip

Because of the previous result, we call \textit{good} a  $\rho_3$-pair of index $r\geqslant 2$.

\begin{corollary}
Let $\G$ be a bipartite $4$-colored graph with a good $\rho_3$-pair. Then either (i) $M_{\G}$ is reducible or (ii) $M_{\G}$ is boundary-reducible  or (iii) $\G$ is not minimal.
\end{corollary}

\begin{proof} Suppose the switching of the  $\rho_3$-pair disconnects $\G$ into two components $\G_1$ and $\G_2$, then either 
$M_{\G}=M_{\G_1}\#M_{\G_2}$ or $M_{\G}=M_{\G_1}\#_{\partial}M_{\G_2}$. If the connected sum is not trivial, then either (i) or (ii) holds. 
If the connected sum is trivial then (iii) holds. Otherwise, suppose the switching does not disconnect $\G$. If $\G'$ is the graph obtained 
after the switching, then either $M_{\G}=M_{\G'}\#(S^2\times S^1)$ or   $M_{\G}=M_{\G'}\#_{\partial} (D^2\times S^1)$ or 
$M_{\G}=M_{\G'}\# (D^2\times S^1)$  or $M_{\G}$ is boundary-reducible. If the connected sum is not trivial then either (i) or (ii) holds, 
otherwise either $M_{\G}=S^2\times S^1$, which is reducible, or   $M_{\G}=D^2\times S^1$, which is boundary-reducible.
\end{proof} 

A $4$-colored graph without good $\rho_2$- and  $\rho_3$-pairs will be called {\it rigid}. 

\begin{corollary} \label{rigid}
Any minimal $4$-colored graph representing a compact, orientable, irreducible and boundary-irreducible $3$-manifold is rigid.\footnote{The 
representation with 6 vertices of $D^2\times S^1$ given in Table~3 is minimal, but it is not rigid since it contains a good $\rho_3$-pair.}
\end{corollary}

Therefore, for the study of compact irreducible and boundary-irreducible orientable $3$-manifolds, it is enough to consider only contracted rigid $4$-colored bipartite graphs.

\subsection{Orientable $3$-manifolds with toric boundary} \mbox{}

We focus on this case since it contains complements of knots/links in the $3$-sphere, which are the most important class of $3$-manifolds with 
boundary. In this context, prime manifolds are irreducible and boundary-prime, and the only prime manifold which is boundary-reducible is 
$D^2\times S^1$.

\begin{theorem} \ \ \ 
\begin{itemize}
\item There exist exactly $200$ non-isomorphic contracted bipartite $4$-colored graphs (without $2$-dipoles) of order $\leqslant 12$ (see the first row in Table~\ref{tab:2}), representing compact $3$-manifolds with (possibly disconnected) toric boundary, and $106$ of these graphs are rigid (see the third row in Table~\ref{tab:2}). 
\item There exist exactly $124$ non-isomorphic contracted bipartite $4$-colored graphs (without $2$-dipoles) of order $\leqslant 16$ (see the second row in Table~\ref{tab:2}), representing compact $3$-manifolds with connected toric boundary, and only $3$ of these graphs are rigid (see the fourth row in Table~\ref{tab:2}).
\end{itemize}
\end{theorem}

Table~\ref{tab:2} lists in the last two rows the number of non-isomorphic contracted rigid $4$-colored graphs representing compact orientable 
$3$-manifolds with toric boundary, indicated by $C^{(2p)}_{rt}$, and the special cases of connected toric boundary, indicated by 
$C^{(2p)}_{rtc}$, showing the wide simplification due to the property of rigidity, expecially in the case of connected boundary.

%\medskip

\begin{table} [!htp]
\caption{Bipartite $4$-colored graphs with $2p$ vertices and no dipoles representing orientable $3$-manifolds with toric boundary}
\label{tab:2}
\begin{tabular}{|c||c|c|c|c|c|c|c|c|}
  \hline
   {\bf 2p }  & 6 & 8 & 10 & 12 & 14 & 16\\
 \hline\hline \ & \ & \ & \ & \ & \ &\  \\
  {\bf $ C^{(2p)}_t$} &  2 & 4 & 20 & 174 & 1979 & 24058\\
\hline \ & \ & \ & \ & \ & \ & \  \\
 {\bf $C^{(2p)}_{tc}$} &  1 & 0 & 0 & 26  & 13 & 84 \\
 \hline \ & \ & \ & \ & \ & \ &\  \\
  {\bf $ C^{(2p)}_{rt}$} &  1 & 4 & 8 & 93 & 1391 & 4695\\
\hline \ & \ & \ & \ & \ & \ & \  \\
 {\bf $ C^{(2p)}_{rtc}$} &  0 & 0 & 0 & 1  & 0 & 2 \\
   \hline \end{tabular}
\end{table}

\bigskip

%%% ---------------------------------------------------------------------

\section{$3$-Manifolds with toric boundary represented by $4$-colored graphs of order $\leqslant 12$}\label{sec4}

The next theorem summarizes the results of computer classification of prime orientable $3$-manifolds with toric boundary represented by $4$-colored graphs of order $\leqslant 12$.

These results were obtained by means of programs for the study of $3$-manifolds as described in Section~\ref{sec1}.

\begin{theorem}\ \ \  
\begin{itemize}
\item There exist exactly $24$ non-homeomorphic  compact orientable prime $3$-manifolds with (possibly disconnected) toric boundary, admitting a representation by a $4$-colored graph of order $\leqslant 12$, and exactly $22$ of them are complements of links in the $3$-sphere (see Table~\ref{tab:3}).
\item There exist exactly $\;4\;$ non-homeomorphic compact orientable prime $3$-manifolds with connected toric boundary, admitting a representation by a $4$-colored graph of order $\leqslant 16$, and exactly $2$ of them are complements of knots in the $3$-sphere (the trivial and the trefoil knot).
\end{itemize}
\end{theorem}

\medskip

Let us describe which kind of $3$-manifolds can be found in Table~\ref{tab:3}. It is worthwhile establishing a standardized nomenclature for the manifolds in the census so that they may be easily referred to. We use a system analogous to that used in the knot and link tables. We sort all $3$-manifolds in the census associated to the rigid graphs with $N$ vertices in arbitrary order (with the exception of $6^1_1$). We use the notation $N_n^k$ for the $n$th manifold constructed from a rigid graph with $N$ vertices on the list with k boundary components.

\subsection{Seifert manifolds}\label{Seifert} \mbox{}

If $F$ is a compact surface with non-empty boundary, $k\in\mathbb N$ and $\{(p_i,q_i)\}_{i=1}^k$ are coprime pairs of integers, with $p_i\geqslant 2$, we define an orientable Seifert manifold $M = \big(F, (p_1, q_1), \ldots, (p_k, q_k)\big)$ as follows. Consider the surface $\Sigma$ obtained from $F$ by removing the interiors of $k$ disjoint discs. The boundary circles of these discs are denoted by $c_1, \ldots, c_k$. Let $c_{k+1}, \ldots , c_n$ be all the remaining circles of $\partial \Sigma$. Let $W$ be the orientable $S^1$-bundle over $\Sigma$. In other words, $W = \Sigma \times S^1$ or $W = \Sigma \widetilde{\times} S^1$, depending on whether $\Sigma$ is orientable or not. We choose an orientation of $W$ and a section $s: \Sigma \to W$ of the projection map $p: W \to \Sigma$. Each torus $T_i = p^{-1}(c_i)$, for $1\leqslant i\leqslant n$, carries an orientation induced by the orientation of~$W$. On each~$T_i$ we choose a coordinate system, taking $s(c_i)$ as the meridian~$\mu_i$ and a fiber $p^{-1}(\{*\})$ as the longitude~$\lambda_i$. The orientations of the coordinate curves must satisfy the following conditions:
\begin{enumerate}
\item If $W = \Sigma \times S^1$, the orientation of each $\lambda_i$ must be induced by a fixed orientation of $S^1$. If $W = \Sigma \widetilde{\times} S^1$, then the orientation of each $\lambda_i$ can be chosen arbitrarily.
\item The intersection number of $\mu_i$ with $\lambda_i$ in~$T_i$ is $1$.
\end{enumerate}
%Now, let us attach solid tori $V_i = D^2 \times S^1$, $1\leqslant i \leqslant k$, to $W$ via homeomorphisms $h_i : \partial V_i\to T_i$  such that each $h_i$ takes the meridian $\partial D^2\times\{*\}$ of $V_i$ into a curve of type $({p_i}, {q_i})$. 
Now, for each $1\leqslant i \leqslant k$, let us attach a solid torus $V_i = D^2 \times S^1$ to $W$ via a homeomorphism $h_i : \partial V_i\to T_i$  taking the meridian $\partial D^2\times\{*\}$ of $V_i$ into a curve of type $({p_i}, {q_i})$. 
The resulting manifold is denoted by  $M = \big(F, (p_1, q_1), \ldots, (p_k, q_k)\big)$, and the cores $\{0\}\times S^1$ of the solid tori $V_i$ in $M$ are called {\it exceptional fibers} of the Seifert manifold. We emphasize that the remaining boundary tori $T_i$ of $M$, for $k+1\leqslant i \leqslant n$, still possess coordinate systems $(\mu_i, \lambda_i)$.

There are seven Seifert manifolds in our census: $6^1_1$, $6^2_1$, $8^3_1$, $10^2_1$, $12^1_1$, $12^2_1$ and $12^4_1$. We denote in Table~\ref{tab:3} by $D^2_i$ and $M^2_i$ the disc and the M\"obius strip with $i$ holes respectively.

\subsection{Graph manifolds}\label{graph-manifolds} \mbox{}

Manifolds which can be obtained from Seifert manifolds by gluing their boundary tori are known as graph manifolds of Waldhausen. The structure of the seven graph manifolds $8^4_1$, $10^3_1$, $12^4_2$, $12^4_3$, $12^4_4$, $12^4_5$ and $12^5_1$ arising in our census is very simple: %There are only two types of manifolds of this kind.
each of them is obtained by gluing together either two or three Seifert manifolds as follows. 

%\begin{enumerate}
% \item Two Seifert manifolds glued together. Each of them is either $D^2_i\times S^1$, $2\leqslant i\leqslant 3$, or $\seifuno {D^2_1}21$.
% \item Three Seifert manifolds glued together. Each of them is fibered over a disc $D^2_i$, $1\leqslant i\leqslant 2$, with $i$ holes and it has at most one exceptional fiber.
%\end{enumerate}

% Now we describe notations for our graph manifolds. 

\begin{enumerate}
 \item Let $M, M'$ be two Seifert manifolds with fixed coordinate systems on $\partial M, \partial M'$, where $\partial M\neq \emptyset, \partial M'\neq \emptyset$. Choose arbitrary tori $T$ and $T'$ of $\partial M$ and $\partial M'$, respectively. We say that a homeomorphism $f_A: T\to T'$ corresponds to a matrix $A=(a_{ij})\in GL_2(Z)$  if $f_A$ takes any curve of type $(m, n)$ on $T$ to a curve of type $(a_{11}m+a_{12}n, a_{21}m+a_{22}n)$ on $T'$. We can then define $M\cup_A M'$ as $M\cup_{f_A} M'$. 

In particular, in our census $M$ and $M'$ can only be either $D^2_i\times S^1$, for $2\leqslant i\leqslant 3$, or $\seifuno {D^2_1}21$.
 \item Let $M, M', M''$ be three Seifert manifolds with fixed coordinate systems on $\partial M, \partial M', \partial M''$, where $\partial M\neq \emptyset, \partial M'\neq \emptyset, \partial M''\neq \emptyset$. Choose arbitrary tori: $T$ of $\partial M$, $T'_1$ and $T'_2$ of $\partial M'$ and $T''$ of $\partial M''$. Let $f_A: T\to T'_1$, $f_B: T''\to T'_2$  be homeomorphisms corresponding to the matrices $A, B\in GL_2(Z)$, then we can define $M\cup_A M'\cup_B M''$ as $M\cup_{f_A} M'\cup_{f_B} M''$. 

In particular, in our census each one of $M,\ M'$ and $M''$ is fibered over a disc $D^2_i$, for $1\leqslant i\leqslant 2$, with $i$ holes and it has at most one exceptional fiber.
\end{enumerate} 
 
\subsection{Hyperbolic manifolds} \mbox{}

Let $M$ be a compact $3$-manifold with toric boundary. We say that $M$ is hyperbolic if, after removing the boundary, we get a cusped $3$-manifold, i.e. a noncompact hyperbolic manifold $\operatorname{Int}(M)$ of finite volume.

We call a cusped finite volume hyperbolic $3$-manifold {\it Platonic} if it can be decomposed into ideal Platonic solids. A {\it tetrahedral} or {\it octahedral} manifold is a Platonic manifold 
made up of the corresponding ideal Platonic solids. The census of tetrahedral manifolds with at most $25$ (orientable case) and $21$ (non-orientable case) tetrahedra was provided in \cite{[FGGTV]}. The tetrahedral census and a small version of the octahedral census \cite{[G]} have been incorporated into \verb"SnapPy" \cite{[SnapPy]}.

There are nine hyperbolic manifolds in our census: $8^4_2$, $12^3_1$, $12^4_7$ -- $12^4_{12}$ and $12^5_2$. All of them, except $12^4_{11}$, are contained in the orientable cusped census or in the censuses of Platonic manifolds of \verb"SnapPy". 

In Table~\ref{tab:3} each compact hyperbolic manifold is identified by the notation of its corresponding cusped manifold.

\subsection{Manifold $12^4_6$} \mbox{}

Our census contains a manifold $12^4_6$ which is obtained by gluing together the Seifert manifold $D^2_2\times S^1$ and the hyperbolic manifold $12^3_1$ along some homeomorphisms of their boundary tori.

The program \verb"3-Manifold Recognizer" split the manifold into two pieces, one, 
the Seifert manifold $D^2_2\times S^1$, was identified by the program itself, while \verb"SnapPy" did it for the hyperbolic part. 

%and a compactification of the cusped manifold \verb's776' along some homeomorphisms of their boundary tori. 

\subsection{Knot and link complements} \mbox{}

The classification results up to $12$ vertices are listed in Table~\ref{tab:3}. The last column contains information about knots or links in $S^3$, if any, whose complements are represented by the corresponding $4$-colored graphs (whose codes are shown in the second column). Whenever it was possible to find a prime link whose complement is a given manifold, we indicated it by its name in the Thistlethwaite link table up to $14$ crossings distributed with \verb"SnapPy". 

We remark that the manifolds $8^3_1$, $8^4_1$ and $12^5_1$ are also complements of non-prime links which are the chains with $3$, $4$ and $5$ unknots respectively.

\medskip

Surprisingly, up to 12 vertices no complement of knots in the $3$-sphere appears, except for the complement of the trivial one ($6^1_1$ in Table~\ref{tab:3}), whereas in this context 18 different manifolds are complements of classical links with more than one
component. So we carried on the research for the case of connected toric boundary, in
order to find the complement of some non-trivial classical knot.

With 14 vertices there are exactly 13 contracted $4$-colored graphs without $2$-dipoles representing manifolds with connected toric boundary, but all of them contain good $\rho_3$-pairs and therefore the represented manifolds are either not prime or just solid tori.

With 16 vertices the number of graphs is 84, and among them only two are rigid. The codes of these two graphs are: \\ DABCHEFGHGFEDCBAGCEABHDF  \\ and \\ 
 DABCHEFGHGFEDCBAGHEACBDF, \\
and they both represent the complement of the trefoil knot $3_1$. 
Figure~\ref{fig: trefoil} depicts the first one of these graphs.

\begin{figure*}
\includegraphics[scale=0.65]{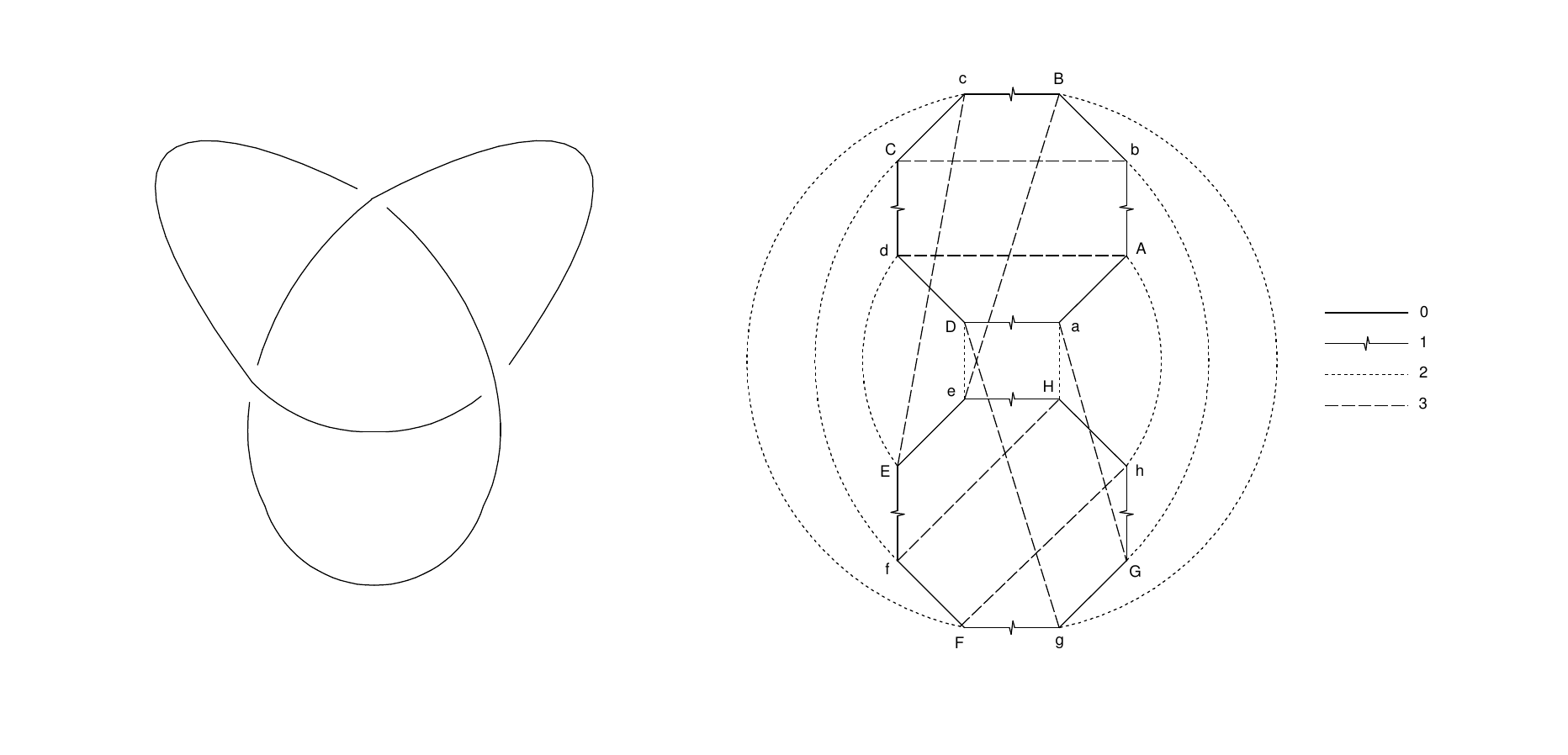}
\caption{The trefoil knot and a graph representing its complement.} 
\label{fig: trefoil}
\end{figure*}

\begin{remark} A construction producing a (non-contracted) $4$-colored graph representing the
complement of  a knot/link $K$ in $S^3$, starting from a connected planar diagram of
$K$, is presented in \cite{Dartois}. However, this construction usually does
not provide minimal representations, even after cancelling proper dipoles and switching good $\rho_2$-pairs. For example, in the case of the complement of the trefoil (resp. the Hopf link) a graph with $18$ vertices (resp. $12$ vertices) is obtained, whereas the minimal one has $16$ vertices (resp. $6$ vertices, as shown in
Table~\ref{tab:3}).
\end{remark}

%------------------------------------------------------------
\begin{table}[h]
\caption{Orientable prime $3$-manifolds with toric boundary represented by $4$-colored graphs of order $\leqslant 12$}
\label{tab:3}
\small{
 \begin{tabular}{|c||c|c|c|c|}
  \hline
   Name & Code & Manifold & Link  \\
 \hline
 \hline 
 $6^1_1$ & CABCBABCA & $D^2\times S^1$ & unknot \\  
 \hline
 $6^2_1$ & CABCABBCA & $D^2_1\times S^1$ & L2a1 \\ 
 \hline
 $8^3_1$ & DABCDCABCADB & $D^2_2\times S^1$ & L6n1 \\  
 \hline
 $8^4_1$ & DABCDCABCBDA & $(D^2_2\times S^1) \bigu 0110 (D^2_2\times S^1)$ & L8n8 \\ 
 \hline
 $8^4_2$ & DABCCDABBCDA & \verb't12047', \verb'ooct02_00001' & L8n7 \\  
 \hline
 $10^2_1$ & EABCDDCEABCDEBA & $\seifuno {D^2_1}21$ & L4a1 \\  
 \hline
 $10^3_1$ & EABCDECDABCDEBA & $\seifuno {D^2_1}21 \bigu 0110 (D^2_2\times S^1)$ & L10n93 \\ 
 \hline
  $12^1_1$ & CABFDEFEDCBAEFABCD & $\seifdue {D^2}2121$ & -- \\  
 \hline
 $12^2_1$ & DABCFEFAECDBBEDFAC & $\seifuno {M^2_1}10$ & -- \\ 
 \hline
 $12^3_1$ & FABCDEEDFBACDFEACB & \verb's776' & L6a5 \\ 
%                &  &  & L12n1951 \\
 \hline
 $12^4_1$ & EABCDFFBEADCEFCABD & $D^2_3\times S^1$ &  see fig. \ref{links2} \\ 
 \hline
 $12^4_2$ & EABCDFFDEACBBEADFC & $\seifuno {D^2_1}21 \bigu 0110 (D^2_3\times S^1)$ & see fig. \ref{links2} \\ 
 \hline
 $12^4_3$ & EABCDFFDAEBCDCEFBA & $(D^2_2\times S^1) \bigu 1211 (D^2_2\times S^1)$ & L14n62853 \\ 
 \hline
 $12^4_4$ & EABCDFFEDABCCDEFAB & $\seifuno {D^2_1}21 \bigu 0110 (D^2_2\times S^1) \bigu 0110 (D^2_2\times S^1)$ & see fig. \ref{links2} \\ 
 \hline
 $12^4_5$ & FABCDEFDAEBCDBEFCA & $(D^2_2\times S^1) \bigu 0110 \seifuno {D^2_1}21 \bigu 0110 (D^2_2\times S^1)$ & see fig. \ref{links2} \\ 
 \hline
 $12^4_6$ & EABCDFFDAEBCCFEBAD & $(D^2_2\times S^1) \bigcup \verb's776'$ & L10n111 \\ 
 \hline
 $12^4_7$ & DABCFEFEABDCEFDACB & \verb'o9_44206' & L10n98 \\ 
 \hline
 $12^4_8$ & DABCFEFDEBACCEAFDB & \verb'ooct03_00011' & L10n100 \\ 
 \hline
 $12^4_9$ & DABCFEFEABDCCDEFAB & \verb'otet10_00014' & L10n101 \\ 
 \hline 
 $12^4_{10}$ & FABCDEDEFABCCDEFAB & \verb'otet10_00028' & L12n2201 \\ 
\hline
 $12^4_{11}$ & DABCFEFDEBACECFADB & hyperbolic manifold with $\operatorname{Vol}=10.6669791338$ & L12n2205 \\ 
 \hline 
 $12^4_{12}$ & CABFDEFCEABDDEACFB & \verb'otet12_00009' & L12n2208 \\ 
 \hline 
 $12^5_1$ & EABCDFFEABDCCDFEBA & $(D^2_2\times S^1) \bigu 0110 (D^2_2\times S^1) \bigu 0110 (D^2_2\times S^1)$ & L14n63765 \\ 
 \hline 
$12^5_2$& DABCFEFEDABCBCFEDA & \verb'otet10_00027' & L10n113 \\ 
%                &  &  & L14n63836 \\
 \hline
\end{tabular}

}

\end{table}

%------------------------------------------------------------

\begin{figure*}
\centering{
\includegraphics[height=0.12\textheight]{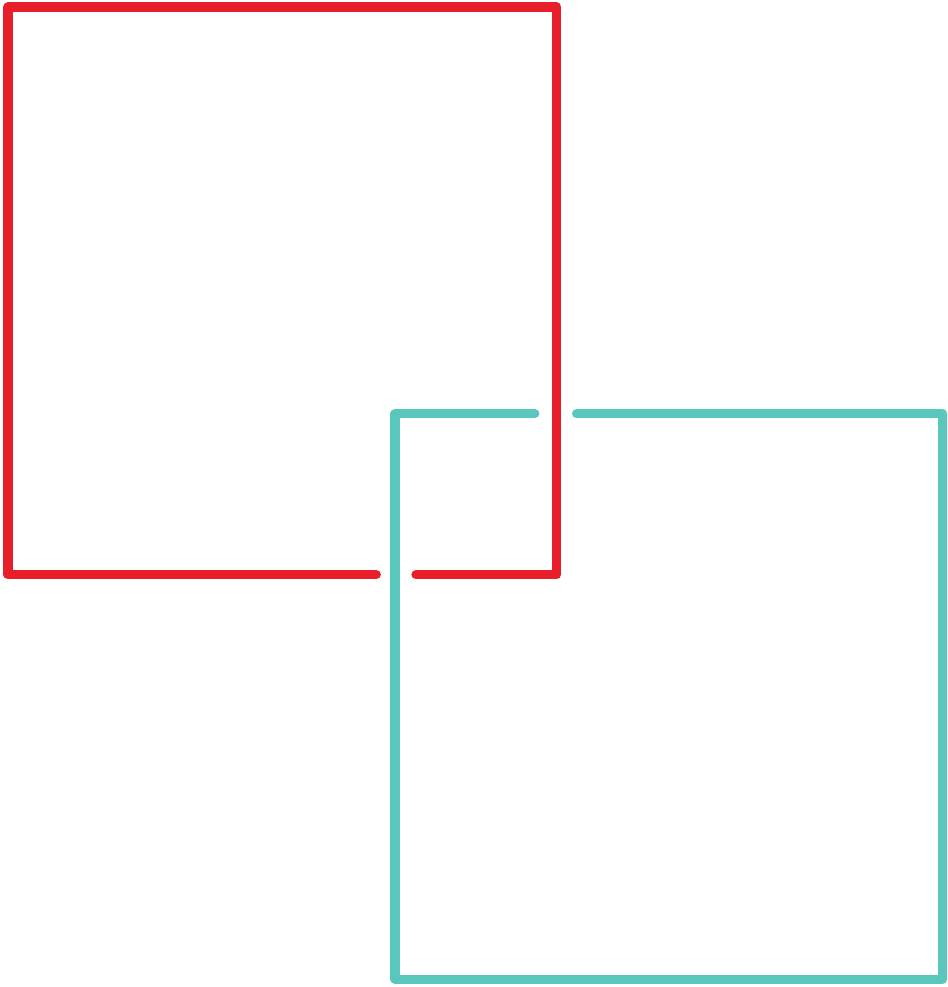}\hspace{10mm}
\includegraphics[height=0.12\textheight]{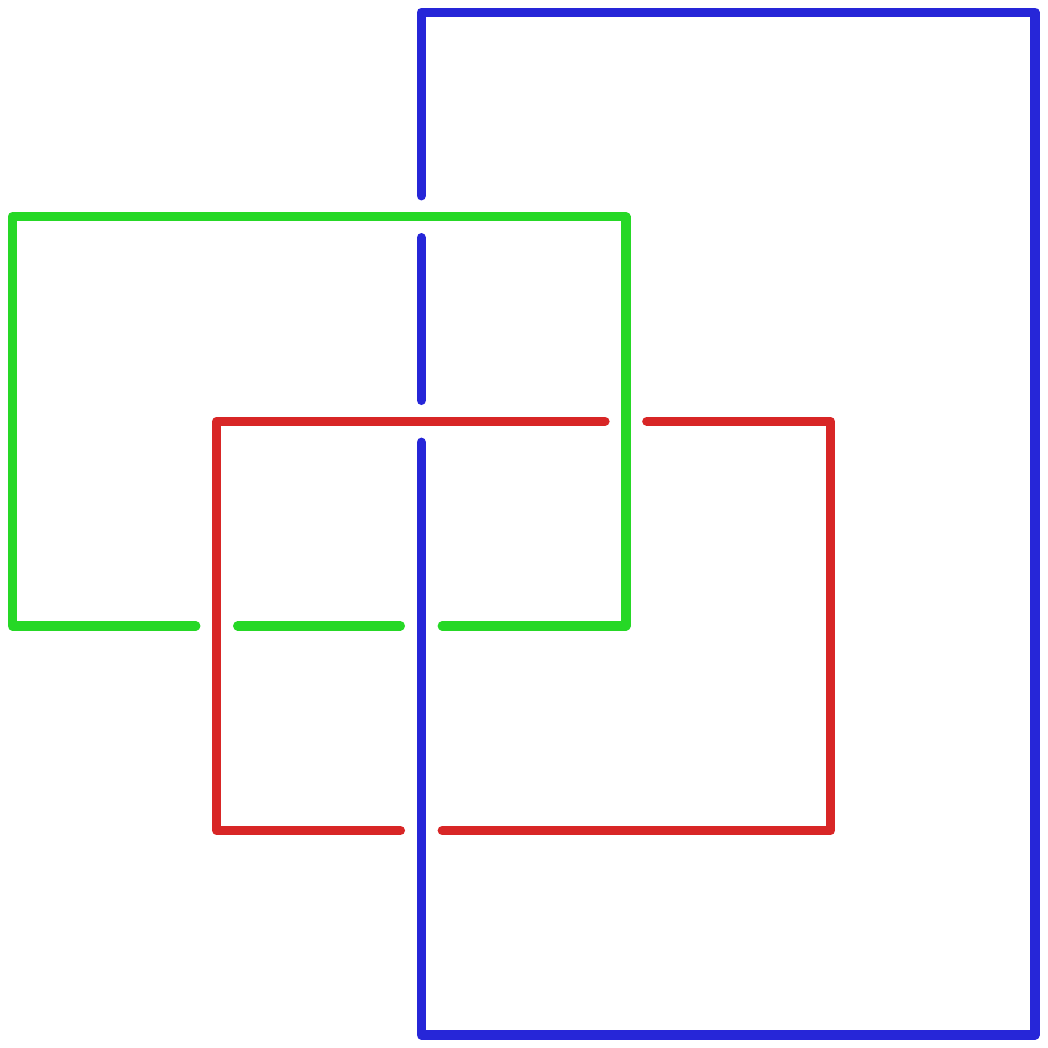}\hspace{10mm}
\includegraphics[height=0.12\textheight]{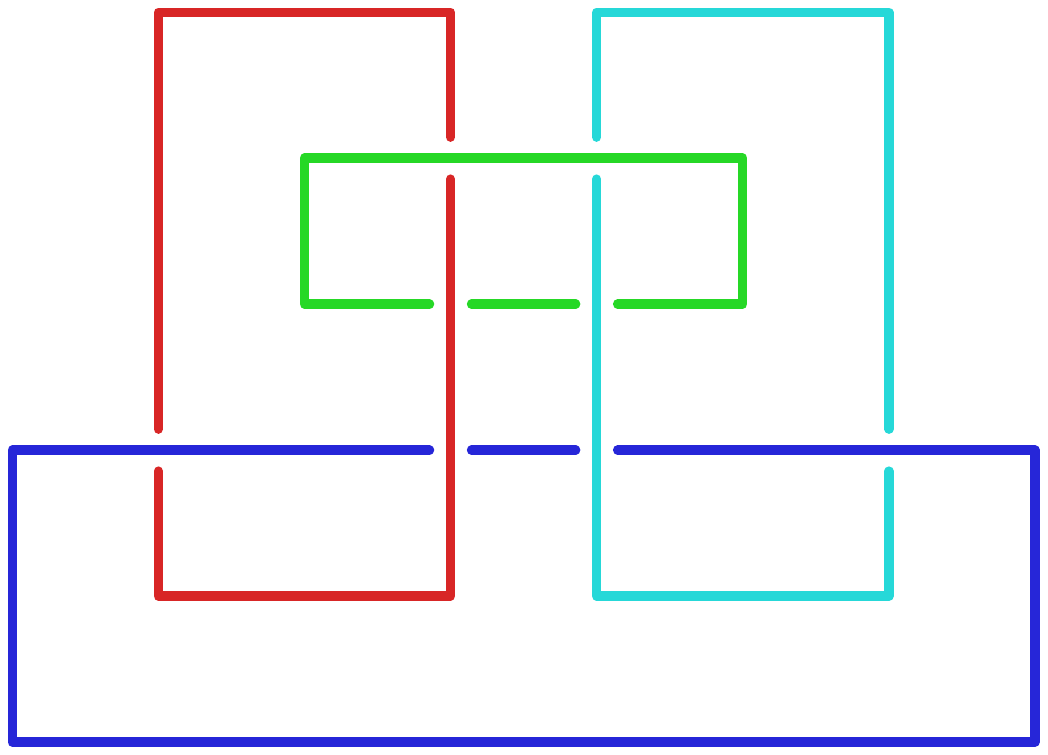}\hspace{10mm}
%\\
$$
6^2_1 \qquad \qquad \qquad \qquad \quad
8^3_1 \qquad \qquad \qquad \qquad \quad
8^4_1 
$$
\medskip
\\
\includegraphics[height=0.12\textheight]{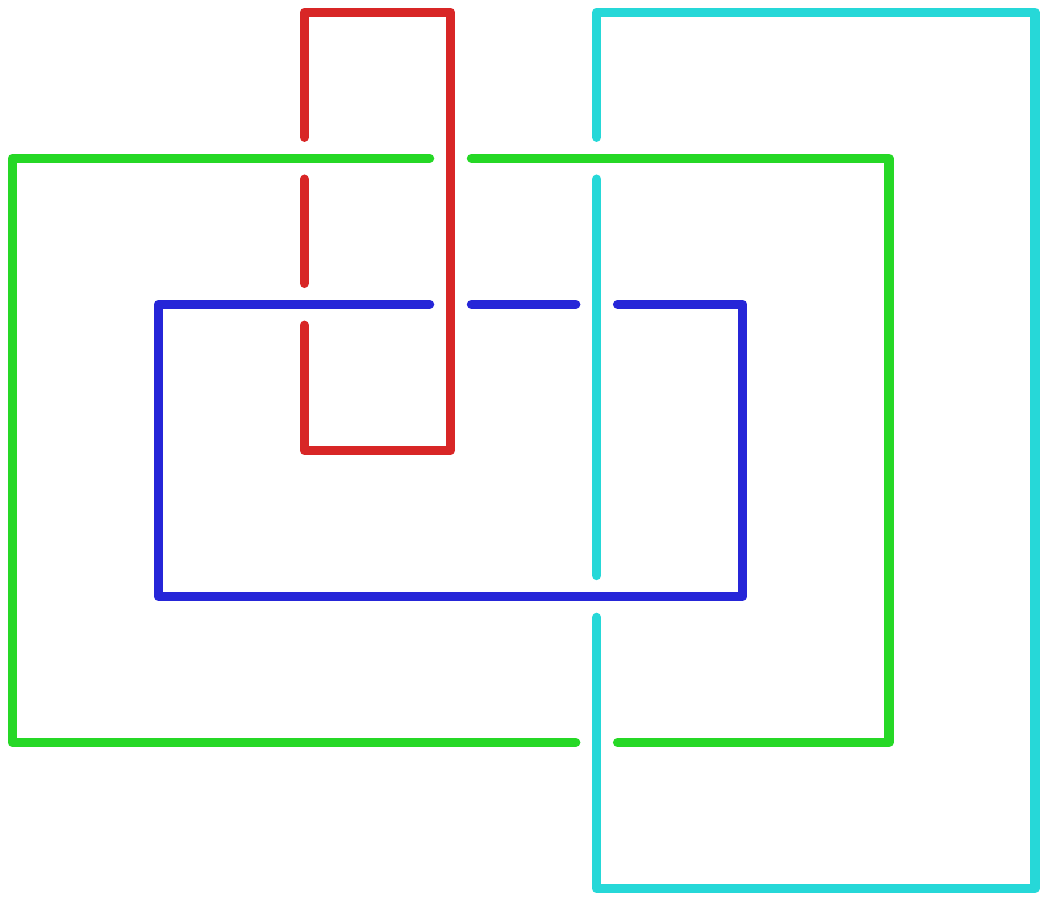}\hspace{10mm} 
\includegraphics[height=0.12\textheight]{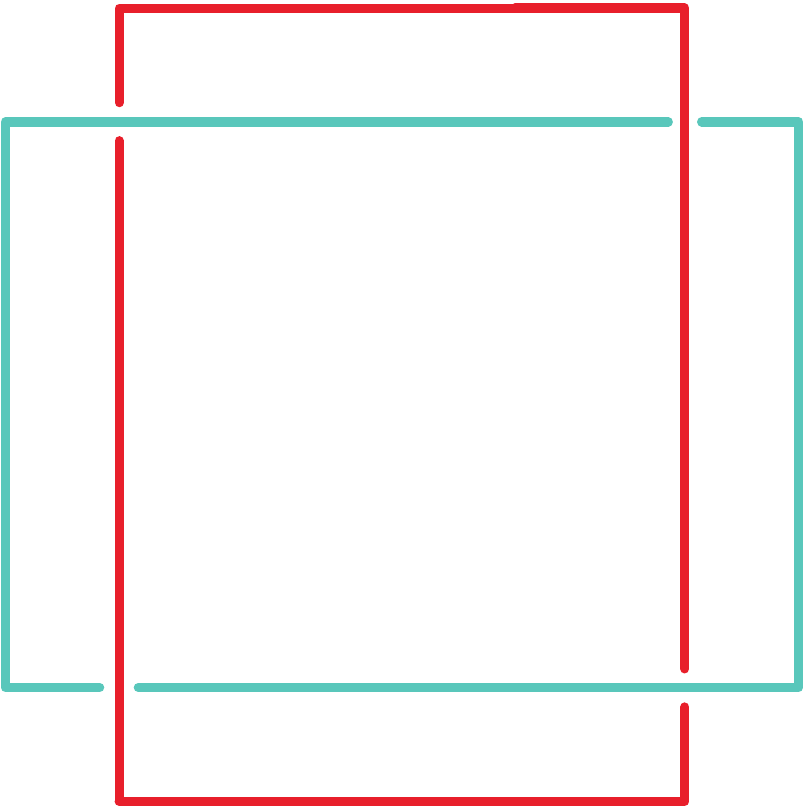}\hspace{10mm}
\includegraphics[height=0.12\textheight]{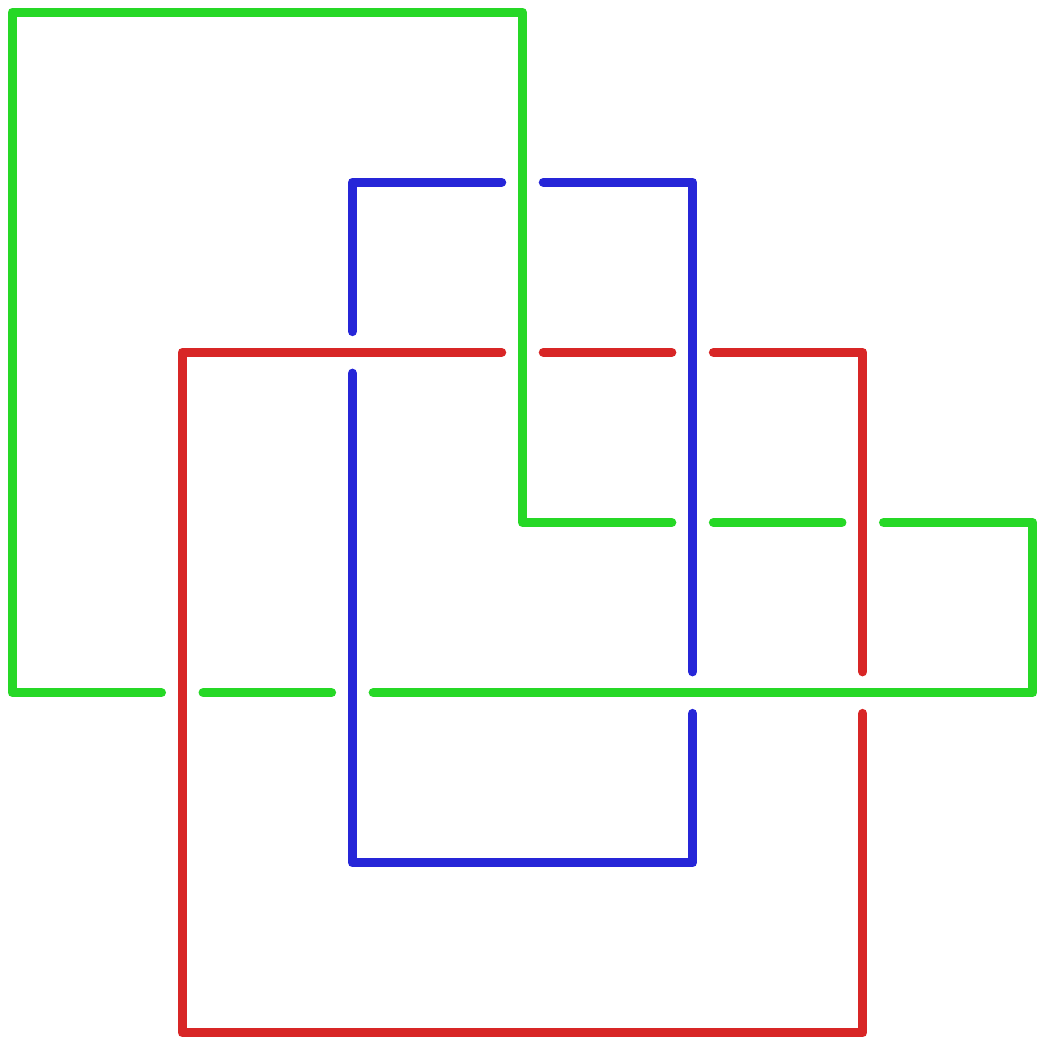}\hspace{10mm}
%\\
$$
8^4_2   \qquad \qquad \qquad \qquad \quad
10^2_1 \qquad \qquad \qquad \qquad \quad
10^3_1
$$
\caption{}
\label{links1}
}
\end{figure*}

\begin{figure*}
\centering{
\includegraphics[height=0.12\textheight]{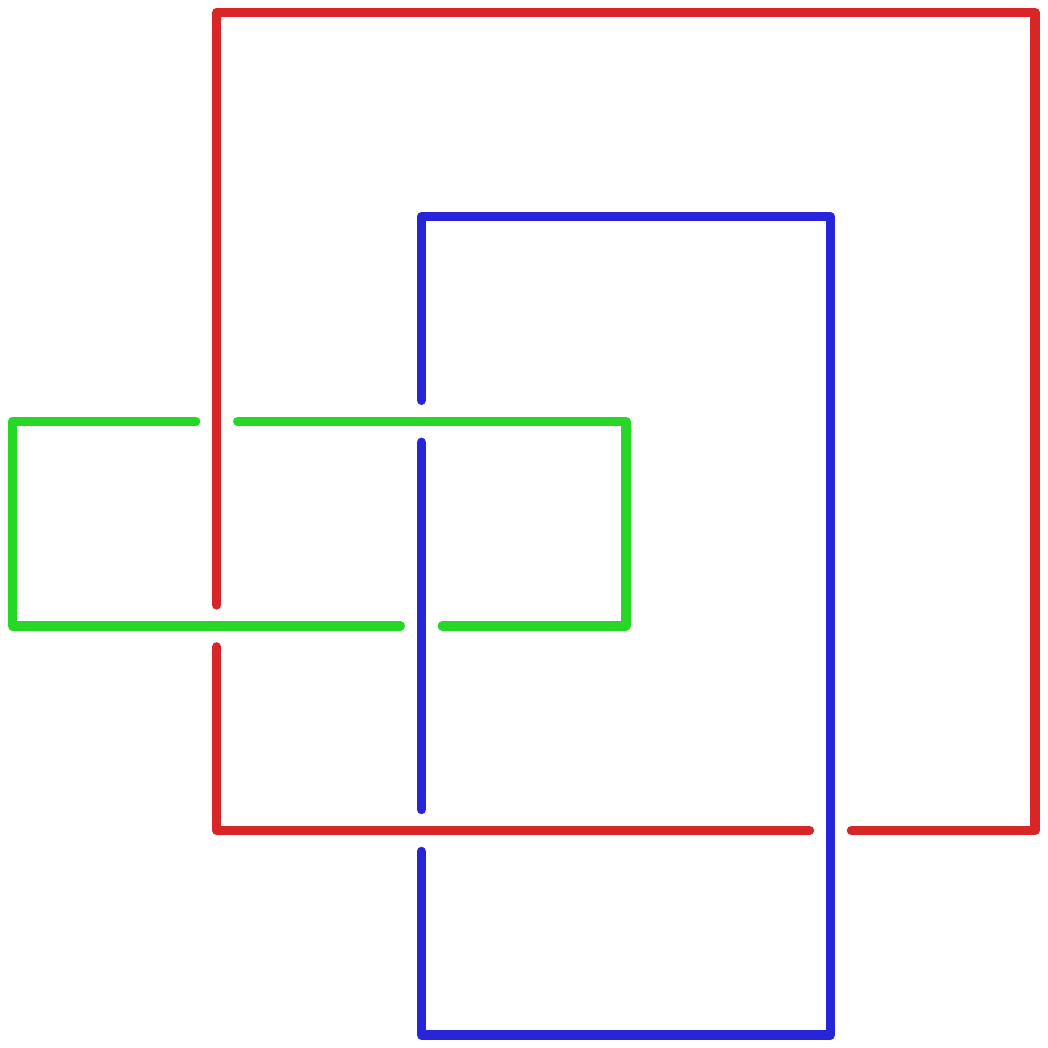}\hspace{10mm} 
\includegraphics[height=0.12\textheight]{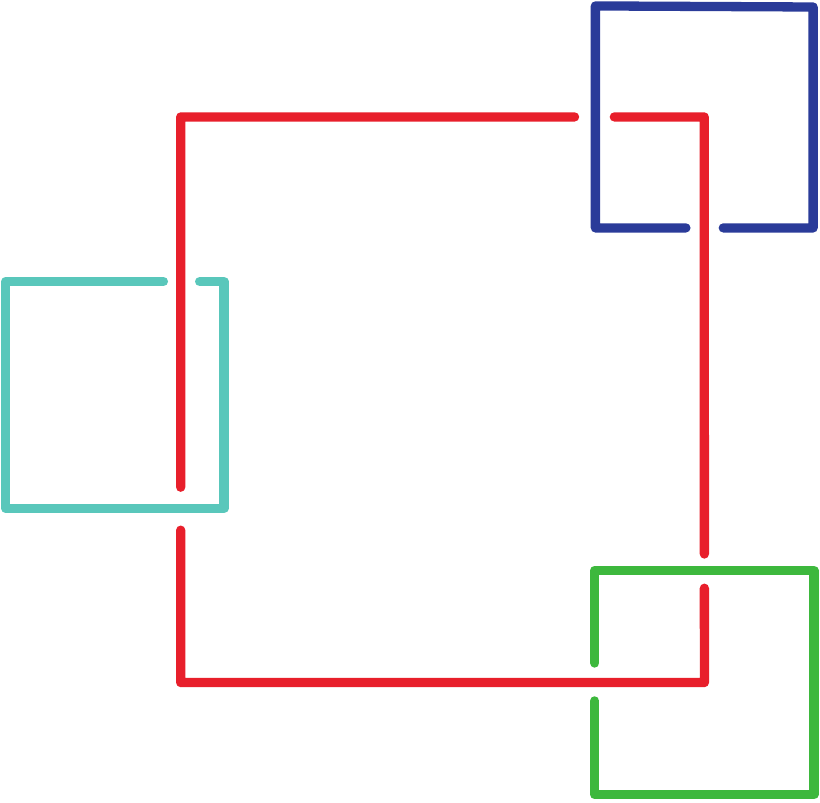}\hspace{10mm}
\includegraphics[height=0.12\textheight]{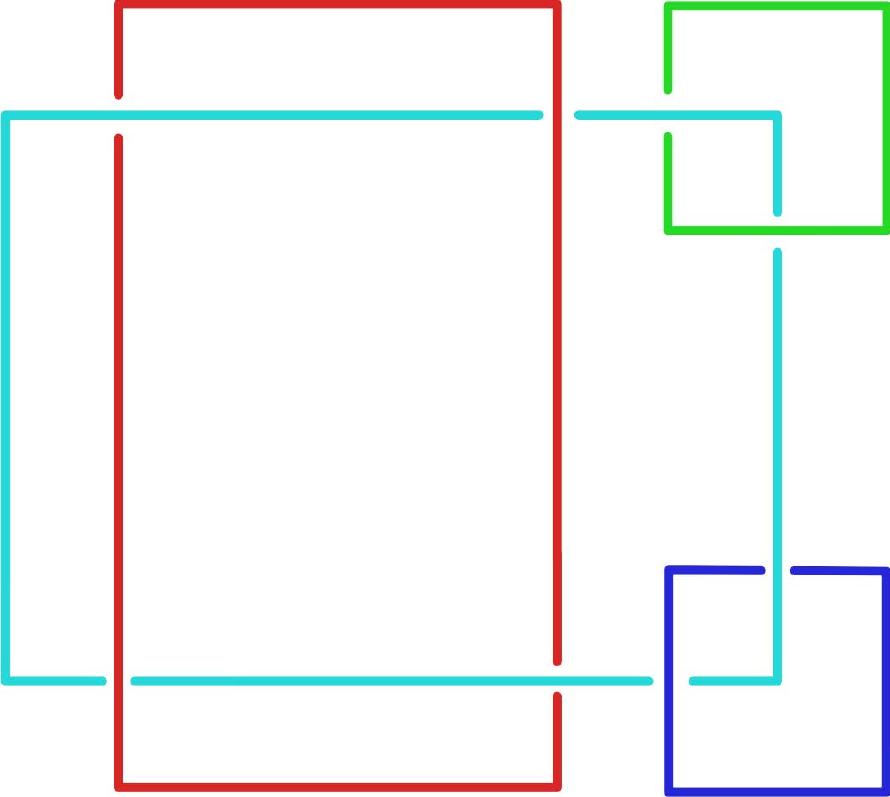}\hspace{10mm}
%\\
$$
12^3_1 \qquad \qquad \qquad \qquad \quad
12^4_1 \qquad \qquad \qquad \qquad \quad
12^4_2 
$$
\medskip
\\
\includegraphics[height=0.12\textheight]{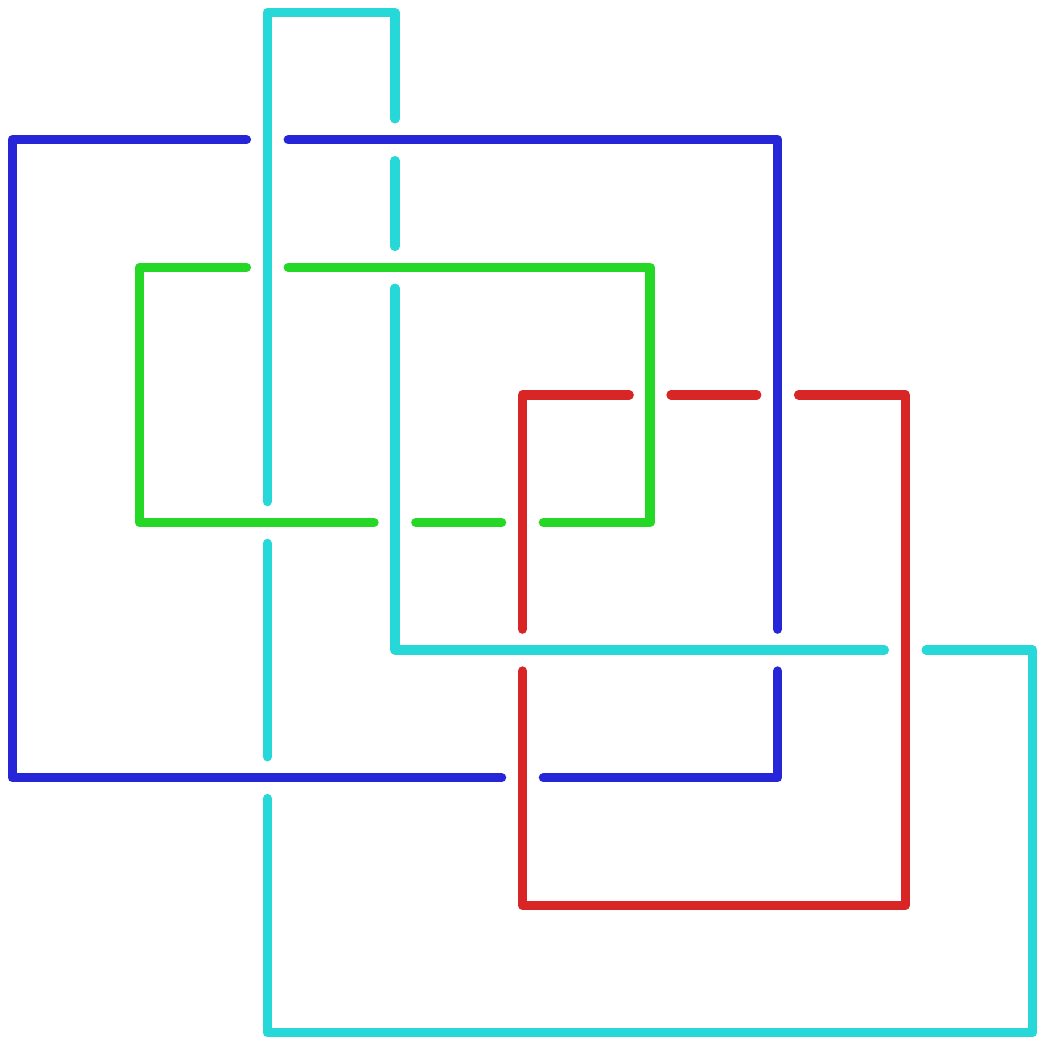}\hspace{10mm}
\includegraphics[height=0.12\textheight]{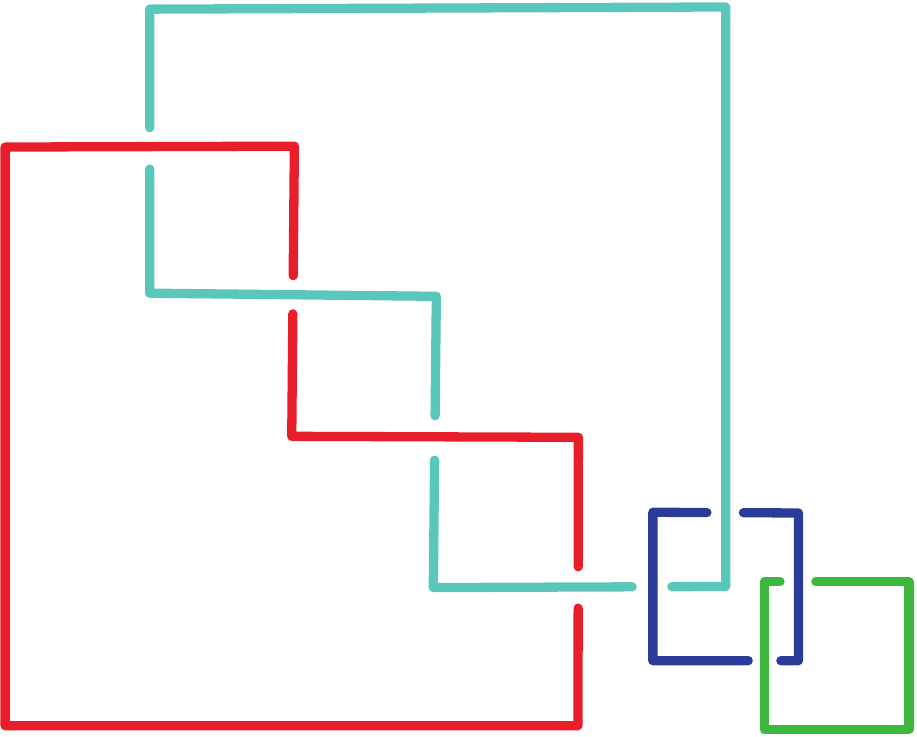}\hspace{10mm}
\includegraphics[height=0.12\textheight]{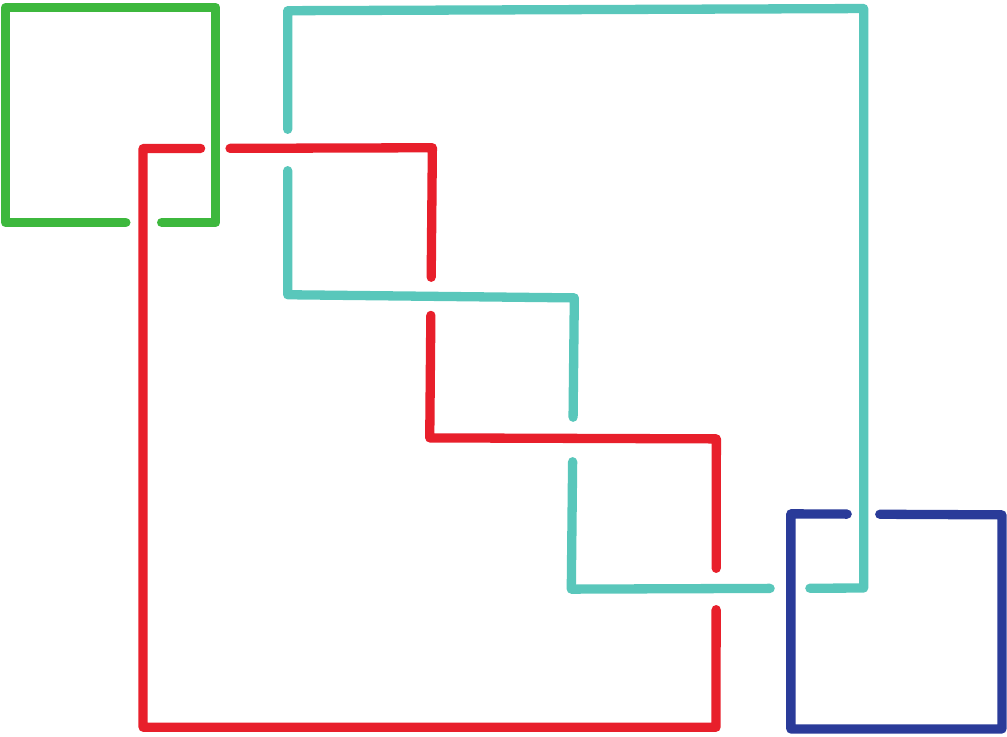}\hspace{10mm}
%\\
$$
12^4_3 \qquad \qquad \qquad \qquad \quad
12^4_4 \qquad \qquad \qquad \qquad \quad
12^4_5
$$
\medskip
\\
\includegraphics[height=0.12\textheight]{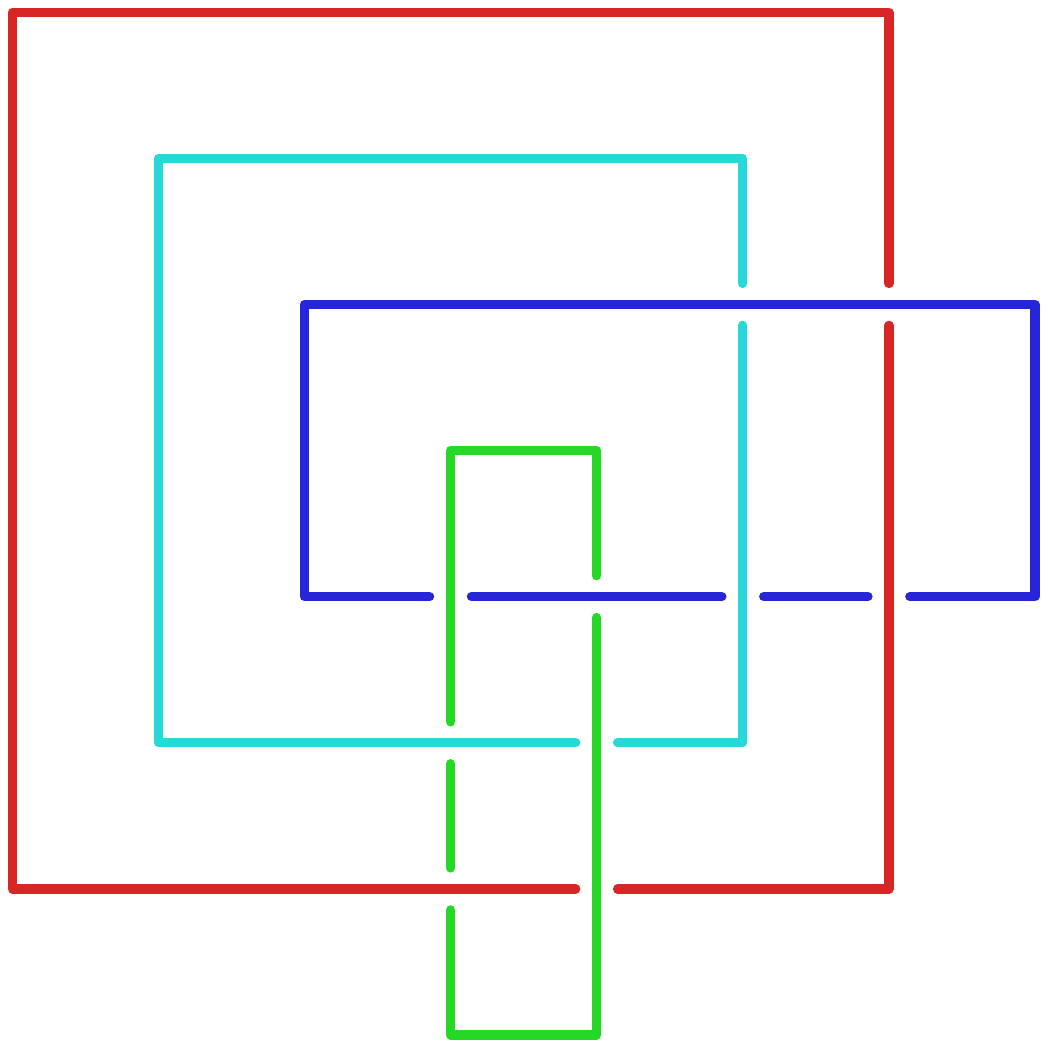}\hspace{10mm} 
\includegraphics[height=0.12\textheight]{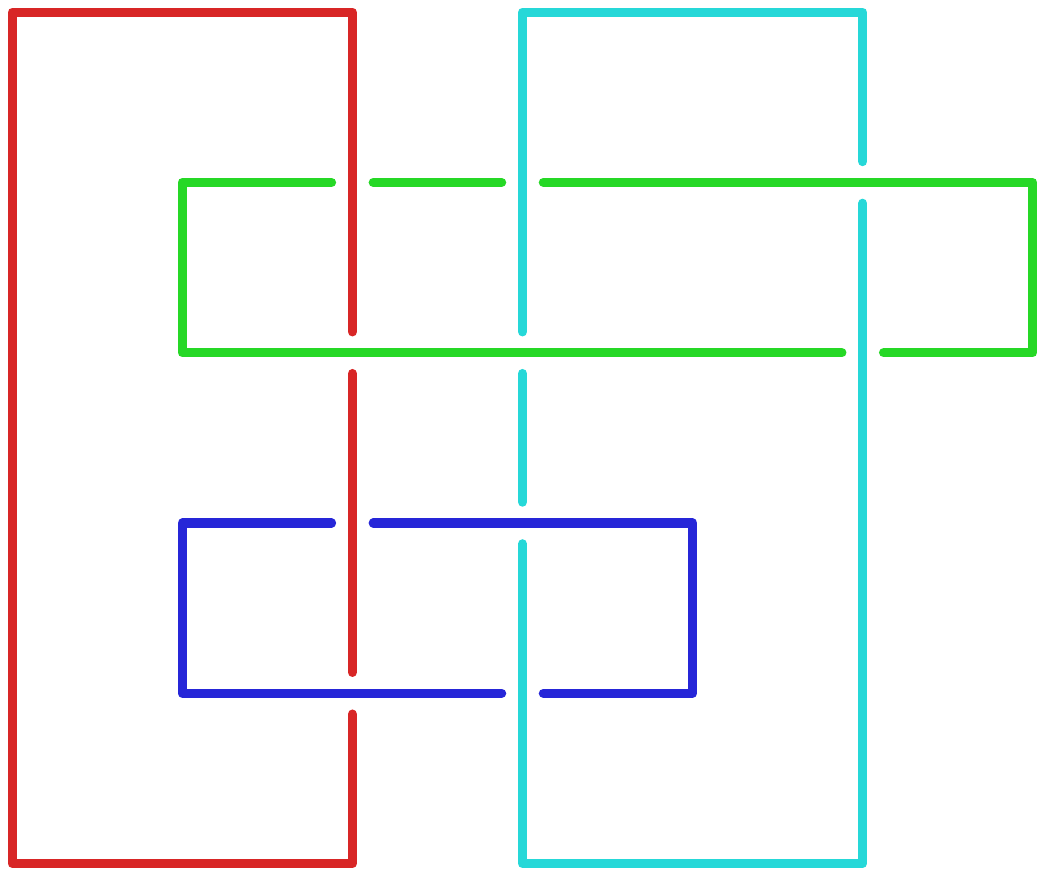}\hspace{10mm}
\includegraphics[height=0.12\textheight]{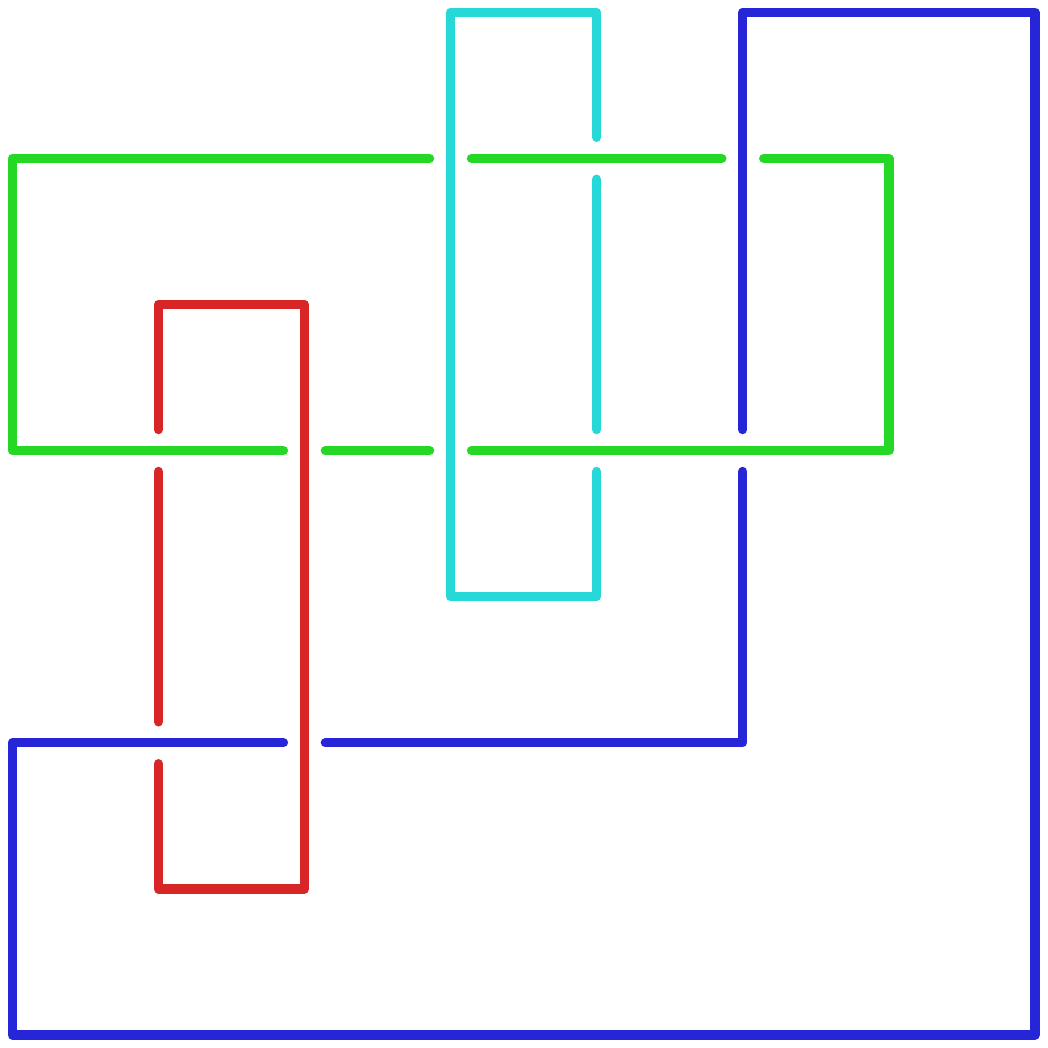}\hspace{10mm}
%\\
$$
12^4_6  \qquad \qquad \qquad \qquad \quad
12^4_7  \qquad \qquad \qquad \qquad \quad
12^4_8 
$$
\\
\includegraphics[height=0.12\textheight]{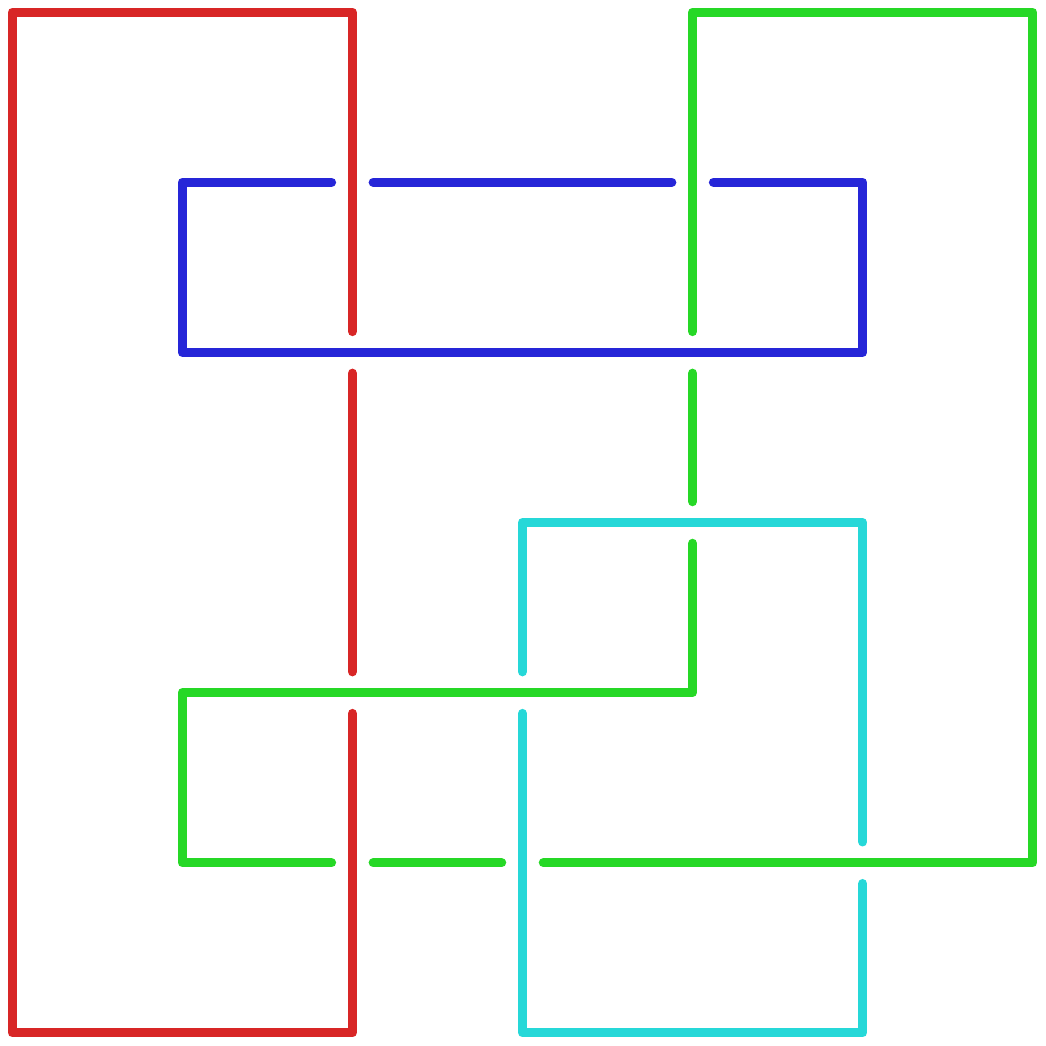}\hspace{10mm} 
\includegraphics[height=0.12\textheight]{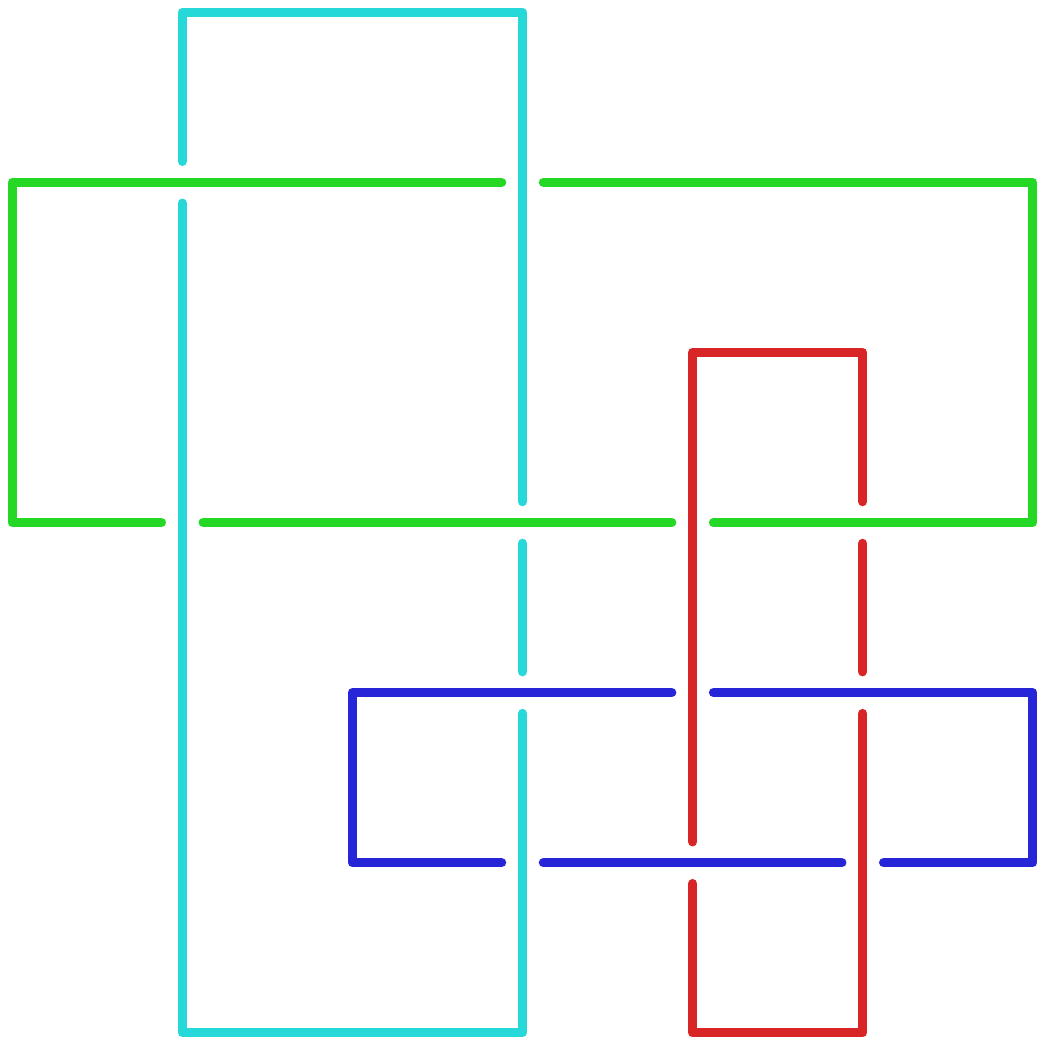}\hspace{10mm}
\includegraphics[height=0.12\textheight]{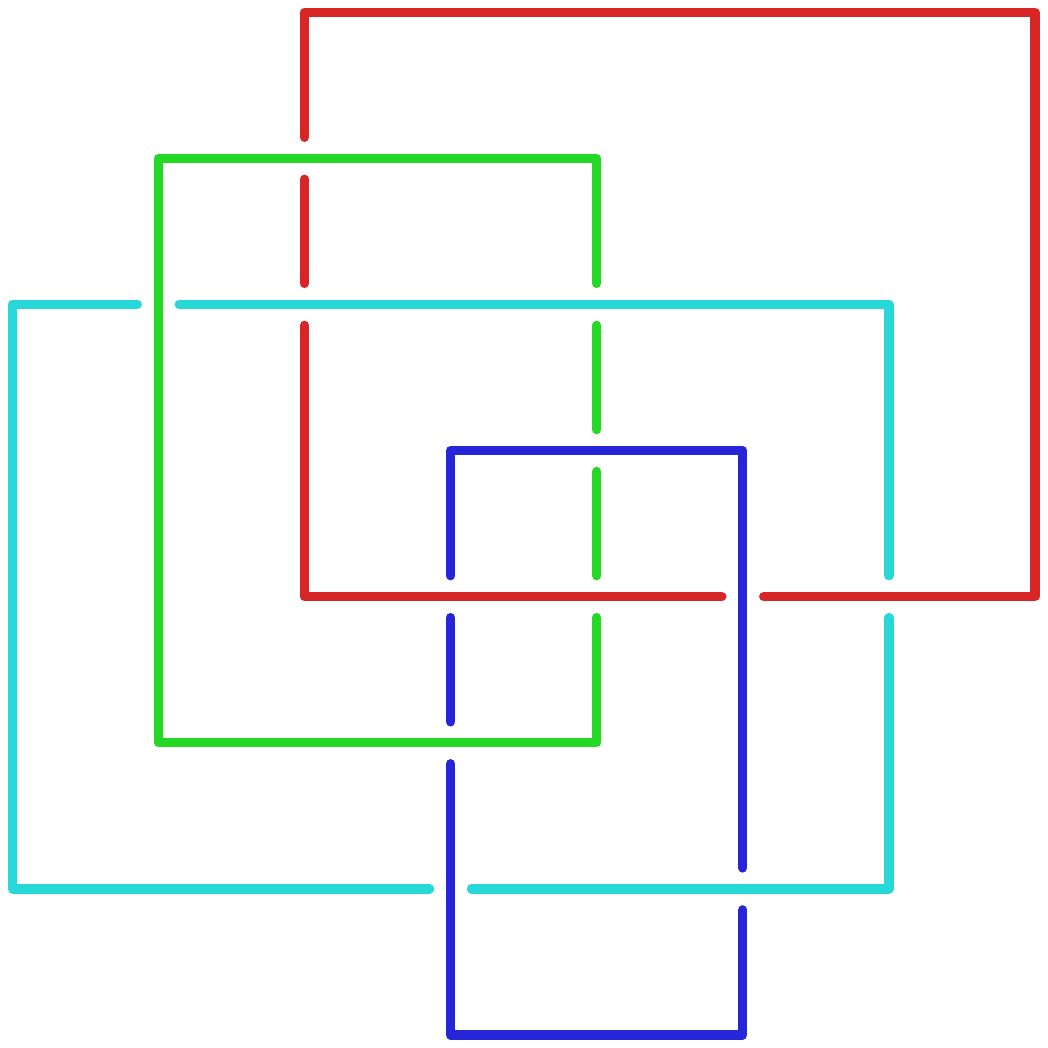}\hspace{10mm}
%\\
$$
12^4_9  \qquad \qquad \qquad \qquad \quad
12^4_{10}  \qquad \qquad \qquad \qquad \quad
12^4_{11} 
$$
\medskip
\\
\includegraphics[height=0.12\textheight]{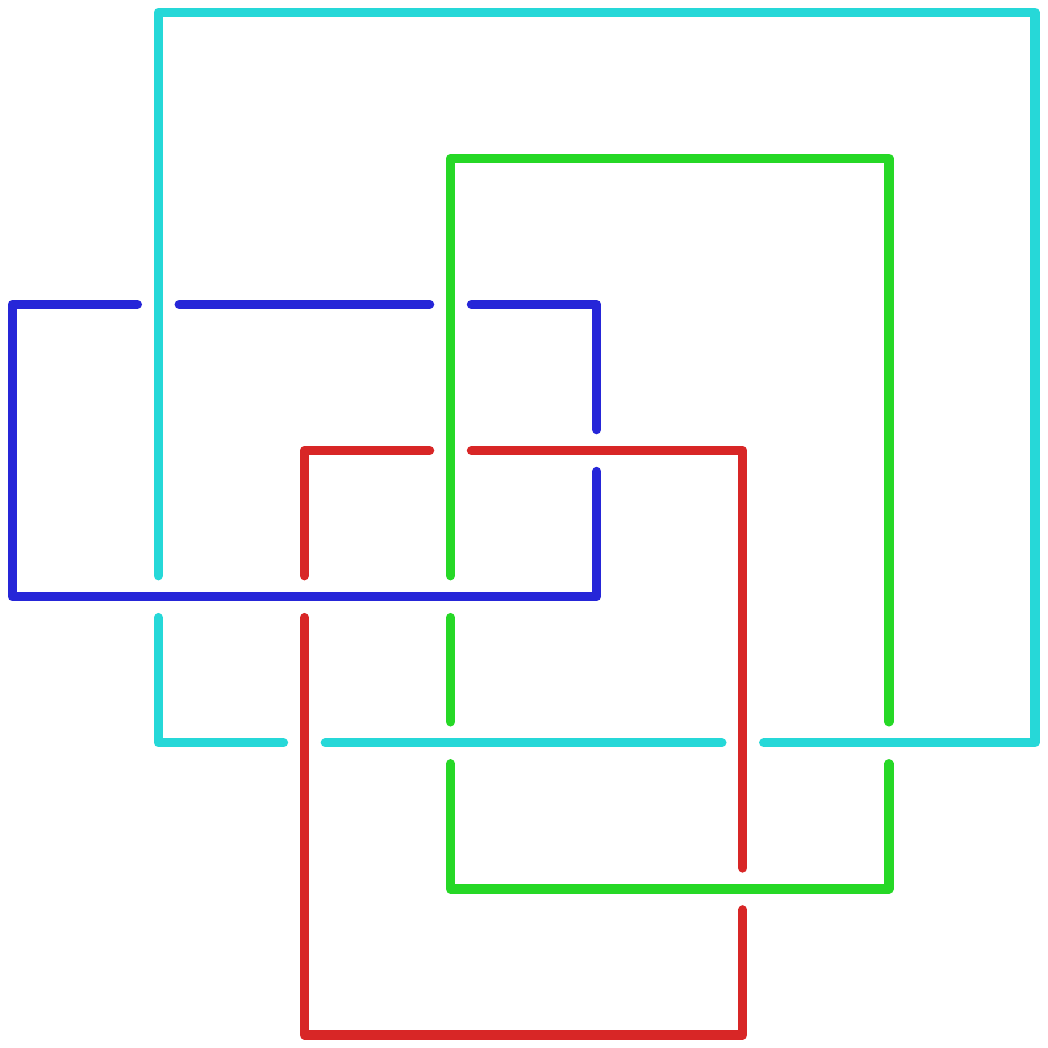}\hspace{10mm} 
\includegraphics[height=0.12\textheight]{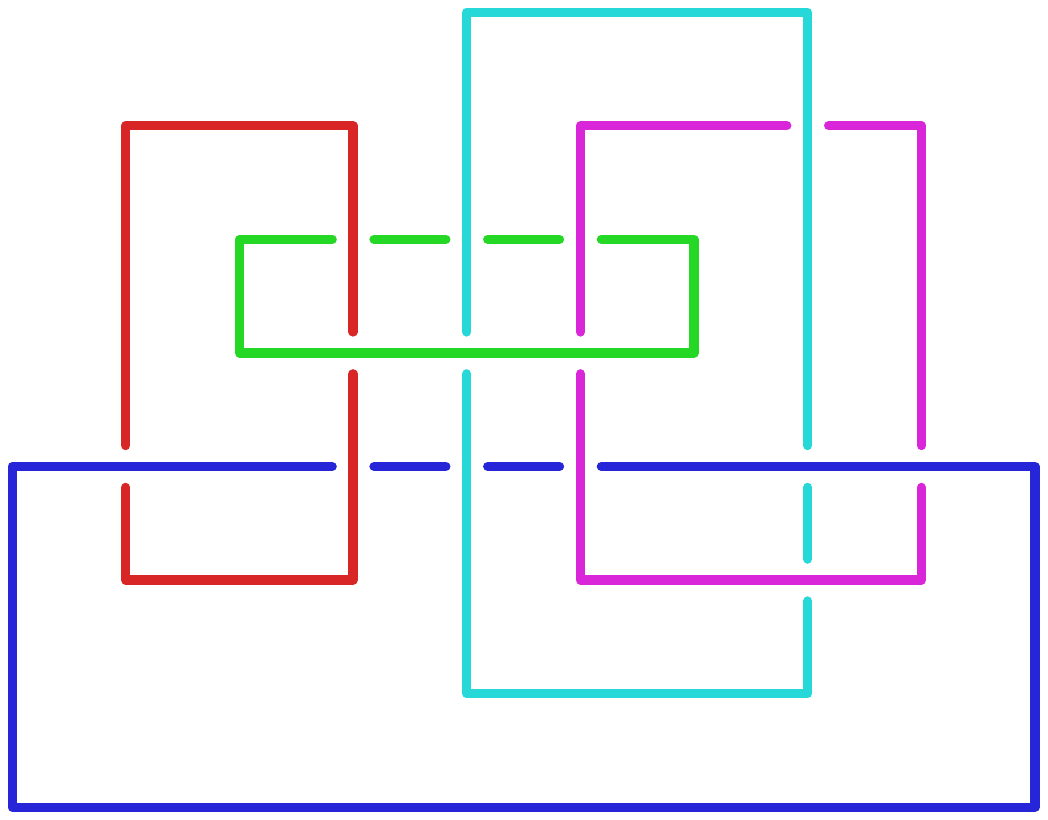}\hspace{10mm}
\includegraphics[height=0.12\textheight]{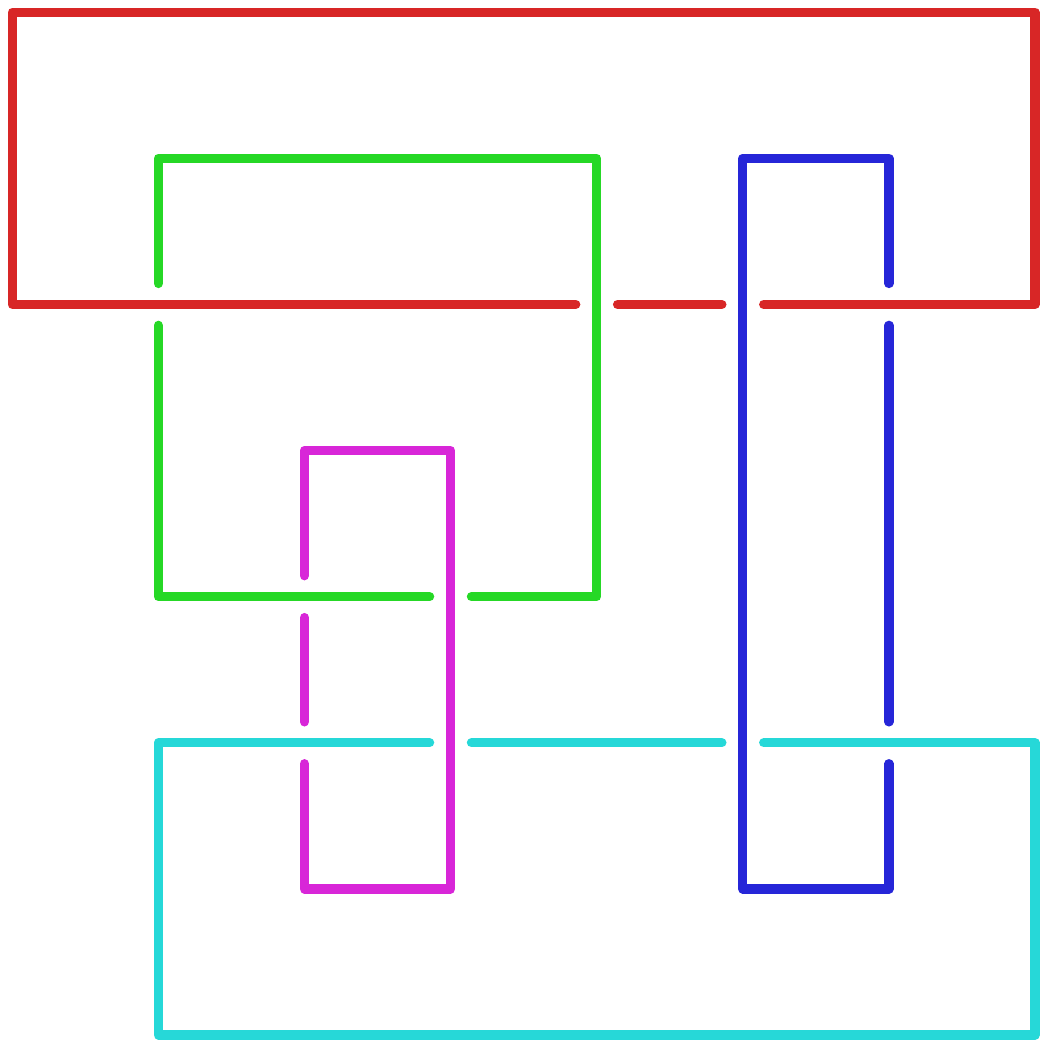}\hspace{10mm}
%\\
$$
12^4_{12}  \qquad \qquad \qquad \qquad \quad
12^5_1  \qquad \qquad \qquad \qquad \quad
12^5_2 
$$
\caption{}
\label{links2}
}
\end{figure*}

\bigskip
\bigskip

\noindent {\it Acknowledgement.} P. Cristofori and M. Mulazzani were supported by GNSAGA of INdAM, University of Modena and Reggio Emilia and
University of Bologna, funds for selected research topics.

\noindent E. Fominykh and V. Tarkaev were supported in part by the Laboratory of Quantum Topology, Chelyabinsk
State University (contract no. 14.Z50.31.0020), the Ministry of Education and Science of
the Russia (the state task number 1.1260.2014/K) and RFBR (grant number 16-01-00609).

\noindent We thank Ekaterina Schumakova for drawing some of the pictures.

\bigskip
\bigskip

%%% ---------------------------------------------------------------------


\begin{thebibliography}{99}
 
%\bibitem{[BCCGM]} P. Bandieri - M. R. Casali - P. Cristofori - L. Grasselli - M. Mulazzani, {\em Computational aspects of crystallization theory: complexity, catalogues  and classification of $3$-manifolds}, Atti Sem. Mat.Fis. Univ. Modena {\bf 58} (2011), 11-45.

\bibitem{[BCrG$_1$]} P. Bandieri - P. Cristofori - C. Gagliardi, {\em Nonorientable $3$-manifolds admitting coloured  triangulations with at most $30$ tetrahedra}, J. Knot Theory Ramifications {\bf 18} (2009), 381--395.

\bibitem{[Burton]} B. Burton, {\em The Pachner Graph and the Simplification of 3-Sphere Triangulations.} In Computational Geometry (SCG'11), pp. 153--162. New York, NY: ACM, 2011. arXiv:1110.6080.

\bibitem{[CHW]}  P. Callahan - M. Hildebrand - J. Weeks, {\em A census of cusped hyperbolic 3-manifolds}, Math. Comp. {\bf 68:225} (1999), 321--332. With microfiche supplement.

\bibitem{[C$_1$]} M. R. Casali, {\em An equivalence criterion for $3$-manifolds},  Rev. Mat. Univ. Complut. Madrid  {\bf 10} (1997), 129--147.

\bibitem{[CC]} M. R. Casali - P. Cristofori,  {\em A catalogue of orientable $3$-manifolds triangulated by $30$ coloured tethraedra}, J. Knot Theory Ramifications {\bf 17} (2008), 579--599.

\bibitem{[CC1]} M. R. Casali - P. Cristofori,  {\em  A note about complexity of lens spaces}, Forum Math. {\bf 27}, Issue 6 (2015) 3173--3188. 

 \bibitem{[CG]} M. R. Casali - C. Gagliardi,  {\em A code for $m$-bipartite  edge-coloured graphs}, dedicated to the memory of Marco Reni, Rend. Ist. Mat. Univ. Trieste {\bf 32} (2001), 55--76.

\bibitem{[CM]} P. Cristofori - M. Mulazzani, {\em Compact 3-manifolds via 4-colored graphs},  Rev. R. Acad. Cienc. Exactas FÃ­s. Nat. Ser. A Math. RACSAM, {\bf 110}(2) (2016), 395--416. DOI 10.1007/s13398-015-0240-8

\bibitem{[SnapPy]} M. Culler - N. Dunfield - J. Weeks, {\em SnapPy, a Computer Program for Studying the Topology of 3-Manifolds}, Available at http://snappy.computop.org, April 2016.

\bibitem{Dartois} S. Dartois, {\em Random Tensor models: Combinatorics, Geometry, Quantum Gravity and Integrability}, Ph.D. Thesis (2015), arXiv:1512.01472

\bibitem{[FGG]} M. Ferri - C. Gagliardi - L. Grasselli, {\em A graph-theoretical representation of PL-manifolds. A survey on crystallizations}, Aequationes Math. {\bf 31} (1986), 121--141.

\bibitem{[FGGTV]} E. Fominykh - S. Garoufalidis - M. Goerner - V. Tarkaev - A. Vesnin, {\em A census of tetrahedral hyperbolic manifolds}, Experimental Math. {\bf 25:4} (2016), 466--481.

\bibitem{[G]} M. Goerner, {\em A census of hyperbolic platonic manifold and augmented knotted trivalent graphs}, arXiv:1602.02208.

%\bibitem{[G]} C. Gagliardi, {\em Regular imbeddings of edge--coloured graphs}, Geom. Dedicata {\bf 11} (1981), 297-314.

\bibitem{[He]} J. Hempel, {\em $3$-manifolds},  Annals of Math. Studies {\bf 86}, pp. 195, Princeton Univ. Press, (1976).

%\bibitem{[HW]} P. J. Hilton - S. Wylie, {\em An introduction to %algebraic topology - Homology theory},  Cambridge Univ. Press, %1960.

%\bibitem{[KnotAtlas]} D. Bar-Natan - S. Morrison et al., {\em The Knot Atlas}, Available at http://katlas.org, February 2016.

\bibitem{[Li]} S. Lins, {\em Gems, computers and attractors for $3$-manifolds}, Knots and Everything {\bf 5}, World Scientific, River Edge, NJ (1995).

\bibitem{[Ma]} S. Matveev, {\em Algorithmic topology and classification of $3$-manifolds}, ACM-Monographs {\bf 9}, Spinger-Verlag, Berlin-Heidelberg-New York (2003).

\bibitem{[Recognizer]} S. Matveev - V. Tarkaev - et al., {\em 3-Manifolds Recognizer}, Available at http://www.matlas.math.csu.ru/?page=recognizer, April 2016.

\bibitem{[Wa]} F. Waldhausen, {\em Eine Klasse von 3-dimensionalen Mannigfaltigkeiten. I, II}, Invent. Math. {\bf 3} (1967), 308--333; ibid. {\bf 4} (1967), 87--117. 

%\bibitem{[Wh]} A. T. White, {\em Graphs, groups and surfaces},  North Holland, 1973.

\end{thebibliography}
\end{document}